\documentclass[12pt,leqno]{article}
\usepackage{amsmath,amssymb,bbm,array,arydshln}
\usepackage{amsthm}
\usepackage{titlesec} 
\usepackage[all,cmtip]{xy}

\titlespacing{\section}{0cm}{3.5pc}{1.5pc}

\makeatletter
\def\@citex[#1]#2{\if@filesw\immediate\write\@auxout{\string\citation{#2}}\fi
  \def\@citea{}\@cite{\@for\@citeb:=#2\do
    {\@citea\def\@citea{\@citesep}\@ifundefined
       {b@\@citeb}{{\bf ?}\@warning
       {Citation `\@citeb' on page \thepage \space undefined}}%
{\csname b@\@citeb\endcsname}}}{#1}}
\def\@citesep{; }
\makeatother

\newtheoremstyle{Kang}{}{}{\itshape}{}{\bf}{}{.5em}{}
\theoremstyle{Kang}
\newtheorem{theorem}{Theorem}[section]

\newtheorem{lemma}[theorem]{Lemma}

\newtheoremstyle{Kremark}{}{}{}{}{\bf}{}{.5em}{}
\theoremstyle{Kremark}

\newtheorem{defn}[theorem]{Definition}

\newtheorem{other}{}
\newenvironment{idef}[1]{\renewcommand{\theother}{#1}\begin{other}}{\end{other}}

\allowdisplaybreaks[1]  

\def\fn#1{\operatorname{#1}} 
\def\bm#1{\mathbbm{#1}}

\def\side#1#2{\mathop{^{#1}\mkern-1.5mu#2}}

\title{Bogomolov Multipliers and Retract Rationality \\ for Semi-direct Products}
\author{Ming-chang Kang \\[3mm]
Department of Mathematics and \\ Taida Institute of Mathematical Sciences,\\
National Taiwan University \\ Taipei, Taiwan \\
E-mail: kang@math.ntu.edu.tw}
\date{}

\begin{document}

\maketitle

\footnote{\textit{\!\!\! Mathematics Subject Classification $(2010)$}: Primary 14E08, 14L30, 13A50, Secondary 12F12, 12F20, 20J06.}
\footnote{\textit{\!\!\! Keywords}:
Noether's problem, the rationality problem, Bogomolov multiplier, unramified Brauer groups, retract rationality, cohomology of finite groups.}
\footnote{\!\!\! Partially supported by National Center for Theoretic Science (Taipei office).}

\begin{abstract}
{\noindent\bf Abstract.} Let $G$ be a finite group. The Bogomolov
multiplier $B_0(G)$ is constructed as an obstruction to the
rationality of $\bm{C}(V)^G$ where $G\to GL(V)$ is a faithful
representation over $\bm{C}$. We prove that, for any finite groups
$G_1$ and $G_2$, $B_0(G_1\times G_2)\xrightarrow{\sim}
B_0(G_1)\times B_0(G_2)$ under the restriction map. If $G=N\rtimes
G_0$ with $\gcd\{|N|,|G_0|\}=1$, then $B_0(G)\xrightarrow{\sim}
B_0(N)^{G_0}\times B_0(G_0)$ under the restriction map. For any
integer $n$, we show that there are non-direct-product $p$-groups
$G_1$ and $G_2$ such that $B_0(G_1)$ and $B_0(G_2)$ contain
subgroups isomorphic to $(\bm{Z}/p \bm{Z})^n$ and $\bm{Z}/p^n
\bm{Z}$ respectively. On the other hand, if $k$ is an infinite
field and $G=N\rtimes G_0$ where $N$ is an abelian normal subgroup
of exponent $e$ satisfying that $\zeta_e\in k$, we will prove
that, if $k(G_0)$ is retract $k$-rational, then $k(G)$ is also
retract $k$-rational provided that certain ``local" conditions are
satisfied; this result generalizes two previous results of Saltman
and Jambor \cite{Ja}.
\end{abstract}

\section{Introduction}

Let $k$ be a field, and $L$ be a finitely generated field
extension of $k$. $L$ is called $k$-rational (or rational over
$k$) if $L$ is purely transcendental over $k$, i.e.\ $L$ is
isomorphic to some rational function field over $k$. $L$ is called
stably $k$-rational if $L(y_1,\ldots,y_m)$ is $k$-rational for
some $y_1,\ldots,y_m$ which are algebraically independent over
$L$. $L$ is called $k$-unirational if $L$ is $k$-isomorphic to a
subfield of some $k$-rational extension of $k$. It is easy to see
that ``$k$-rational" $\Rightarrow$ ``stably $k$-rational"
$\Rightarrow$ ``$k$-unirational".

The L\"uroth problem asks whether a $k$-unirational field is
necessarily $k$-rational. See the paper of Manin and Tsfasman
\cite{MT} for a survey. A special case of the L\"uroth problem is
the following Noether's problem. Let $k$ be any field, $G$ be any
finite group. Consider the regular representation $G\to
GL(V_{\fn{reg}})$ over $k$. Thus $G$ acts on the rational function
field $k(V_{\fn{reg}})=k(x_g:g\in G)$ by $k$-automorphism defined
by $h\cdot x_g=x_{hg}$ for any $g,h \in G$. Denote by $k(G)$ the
fixed field, i.e. $k(G)=k(x_g:g\in G)^G$. Noether's problem asks,
under what situation, the field $k(G)$ is $k$-rational. See Swan's
paper for a survey of Noether's problem \cite{Sw1}.

The Bogomolov multiplier, denoted by $B_0(G)$ which will be
defined in Definition \ref{d1.3}, is an obstruction to the
rationality (resp. the stable rationality) of $k(G)$ when $k$ is
an algebraically closed field satisfying that $\gcd
\{|G|,\fn{char} k\}=1$. Before we define $B_0(G)$, let's recall
the notion of retract rational which evolves from the notion of a
unirational variety $V$ such that the dominating map $\bm{P}^n
\dashrightarrow V$ has a rational section.

\begin{defn}[{\cite[Definition 3.1]{Sa3}}] \label{d1.1}
Let $k$ be an infinite field and $L$ be a field containing $k$.
$L$ is called retract $k$-rational, if there are some affine
domain $A$ over $k$ and $k$-algebra morphisms $\varphi: A\to
k[X_1,\ldots,X_n][1/f]$, $\psi:k[X_1,\ldots, X_n][1/f] \to A$
where $k[X_1,\ldots,X_n]$ is a polynomial ring over $k$, $f\in
k[X_1,\ldots,X_n]\backslash \{0\}$, satisfying that

(i) $L$ is the quotient field of $A$, and

(ii) $\psi\circ \varphi=1_A$, the identity map on $A$.
\end{defn}

It is not difficult to see that ``stably $k$-rational" $\Rightarrow$ ``retract $k$-rational" $\Rightarrow$ ``$k$-unirational".
Moreover, if $k$ is an infinite field, it is known that the following three conditions are equivalent:
$k(G)$ is retract $k$-rational, there is a generic $G$-Galois extension over $k$,
there exists a generic $G$-polynomial over $k$ \cite{Sa1,Sa3,De}.

If $k\subset K$ is an extension of fields, the notion of the
unramified Brauer group of $K$ over $k$, denoted by
$\fn{Br}_{v,k}(K)$, was introduced by Saltman \cite{Sa2}. If $k$
is algebraically closed and $K$ is retract $k$-rational, then
$\fn{Br}_{v,k}(K)=0$ (see \cite{Sa2,Sa4}). Using this criterion,
Saltman was able to construct, for any prime number $p$ and any
algebraically closed field $k$ with $\fn{char}k\ne p$, there is
$p$-group $G$ of order $p^9$ such that $\fn{Br}_{v,k}(k(G))\ne 0$.
In particular, $k(G)$ is not retract $k$-rational and therefore is
not $k$-rational \cite{Sa2}; also see the remark after Theorem
\ref{t5.2} of this paper.

The computation of $\fn{Br}_{v,k}(k(G))$ becomes more effective
because of the following theorem.

\begin{theorem}[{Bogomolov, Saltman \cite[Theorem 12]{Bo,Sa5}}] \label{t1.2}
Let $G$ be a finite group, $k$ be an algebraically closed field with $\gcd\{|G|,\fn{char}k\}=1$.
Let $\mu$ denote the multiplicative subgroup of all roots of unity in $k$.
Then $\fn{Br}_{v,k}(k(G))$ is isomorphic to the group $B_0(G)$ defined by
\[
B_0(G)=\bigcap_A \fn{Ker} \{\fn{res}_G^A:H^2(G,\mu)\to H^2(A,\mu)\}
\]
where $A$ runs over all the bicyclic subgroups of $G$ $($a group
$A$ is called bicyclic if $A$ is a cyclic group or a direct
product of two cyclic groups$)$.
\end{theorem}

Because of the above theorem, we define $B_0(G)$ as follows.

\begin{defn} \label{d1.3}
Let $G$ be a finite group, $\bm{Q}/\bm{Z}$ be the
$\bm{Z}[G]$-module with trivial $G$-actions. Define $B_0(G)$ by
\[
B_0(G)=\bigcap_A \fn{Ker} \{\fn{res}_G^A:H^2(G,\bm{Q}/\bm{Z})\to H^2(A,\bm{Q}/\bm{Z})\}
\]
where $A$ runs over all the bicyclic subgroups of $G$ and $\fn{res}_G^A$ denotes the restriction map from $G$ to $A$.
\end{defn}

The cohomology group $H^2(G,\bm{Q}/\bm{Z})$ is called the Schur
multiplier of $G$ in some literature \cite{Kar}. Thus the group
$B_0(G)$, as a subgroup of $H^2(G,\bm{Q}/\bm{Z})$, was called the
Bogomolov multiplier of $G$ by Kunyavskii \cite{Ku}; we follow his
terminology in this paper.

\medskip
When $N$ is a subgroup of $G$, the restriction map
$\fn{res}:H^2(G,\bm{Q}/\bm{Z})\to H^2(N,\bm{Q}/\bm{Z})$ induces a
well-defined map $\fn{res}:B_0(G)\to B_0(N)$. When $N\lhd G$, we
also have a well-defined map $\fn{res}:B_0(G)\to B_0(N)^G$ (see
Lemma \ref{l2.3}). We will prove the following two theorems.

\begin{theorem} \label{t1.4}
Let $G_1$ and $G_2$ be any finite groups.
Then the restriction map $\fn{res}:B_0(G_1\times G_2)\to B_0(G_1)\times B_0(G_2)$ is an isomorphism.
\end{theorem}

\begin{theorem} \label{t1.5}
Let $G=G_1\rtimes G_2$ be a finite group with $G_1\lhd G$. If
$\gcd\{|G_1|,|G_2|\}=1$, then the restriction map
$\fn{res}:B_0(G)\to B_0(G_1)^{G_2}\times B_0(G_2)$ is an
isomorphism.
\end{theorem}

As an application of Theorem \ref{t1.5}, we will show that, if
$G=N\rtimes G_0$ is a Frobenius group, then $\fn{res}:B_0(G)\to
B_0(N)^{G_0}$ is an isomorphism (see Theorem \ref{t2.8}).

One may suspect that $B_0(G)$ would be ``small" if it is
non-trivial (see the remark at the end of Section 5 of this
paper). However, we will show that, if $n$ is any positive
integer, then there are non-direct-product $p$-groups $G_1$ and
$G_2$ such that $B_0(G_1)$ and $B_0(G_2)$ contain subgroups
isomorphic to $(\bm{Z}/p \bm{Z})^n$ and $\bm{Z}/p^n \bm{Z}$
respectively (see Theorem \ref{t5.2} and Theorem \ref{t5.4}). When
$n=1$, the group constructed in Theorem \ref{t5.2} is the same one
constructed by Saltman and Shafarevich \cite[page 83; Sh, page
245]{Sa2}. For this case, Saltman and Shafarevich show that the
unramified Brauer groups are non-trivial; their proof are
different from our direct proof by showing the non-triviality of
$B_0(G)$.

\medskip
The condition $\gcd\{|G_1|,|G_2|\}=1$ of Theorem \ref{t1.5} is
reminiscent of the following theorem of Saltman.

\begin{theorem}[{Saltman \cite[Theorem 3.1 and Theorem 3.5; Ka1, Theorem 3.5]{Sa1}}] \label{t1.6}
Let $k$ be an infinite field, $G=N\rtimes G_0$ with $N\lhd G$.

{\rm (1)} If $k(G)$ is retract $k$-rational, so is $k(G_0)$.

{\rm (2)} Assume furthermore that $N$ is abelian and $\gcd\{|N|,|G_0|\}=1$.
If both $k(N)$ and $k(G_0)$ are retract $k$-rational, so is $k(G)$.
\end{theorem}

For the convenience of presenting Jambor's Theorem, we give the
following definition.

\begin{defn} \label{d1.7}
Let $G=N\rtimes G_0$ be a finite group where $N$ is an abelian
normal subgroup of $G$. Write $N=\prod_{p\mid |N|} N_p$ where each
$N_p$ is the $p$-Sylow subgroup of $N$. Define $H_p=\{g\in G_0:
g\tau g^{-1}=\tau$ for any $\tau\in N_p\}$.
\end{defn}

With the above definition, we may state Jambor's Theorem, which is
in some sense a local refinement of Saltman's Theorem (see Theorem
\ref{t1.6}).

\begin{theorem}[Jambor \cite{Ja}] \label{t1.8}
Let $G=N\rtimes G_0$ be a finite group where $N$ is an abelian
normal subgroup of $G$. Assume that, for each prime number $p\mid
|N|$, $p$ doesn't divide $[G_0:H_p]$. If $k$ is an infinite field
such that both $k(N)$ and $k(G_0)$ are retract $k$-rational, then
$k(G)$ is also retract $k$-rational.
\end{theorem}

What we will prove is that the condition that $p\nmid [G_0:H_p]$ can be weakened if we assume that $k$ has enough roots of unity.
Namely,

\begin{theorem} \label{t1.9}
Let $G=N\rtimes G_0$ be a finite group where $N$ is an abelian
normal subgroup of exponent $e$. Assume that, for each prime
number $p\mid |N|$, either $p\nmid [G_0:H_p]$ or all the Sylow
subgroups of $G_0/H_p$ are cyclic groups. If $k$ is an infinite
field such that $\zeta_e \in k$ and $k(G_0)$ is retract
$k$-rational, then $k(G)$ is also retract $k$-rational.
\end{theorem}

We remark that our proof of the above theorem, even for the case
of Theorem \ref{t1.8}, is different from Jambor's proof.

The paper is organized as follows. We prove Theorem \ref{t1.4} and
Theorem \ref{t1.5} in Section 2. Note that Theorem \ref{t1.5} is
valid in a more general context. See Theorem \ref{t2.6} (also see
Theorem \ref{t2.7}). Section 3 contains some preliminaries for
proving Theorem \ref{t1.9}. We recall the theory of flabby
lattices in this section. Theorem \ref{t1.9} is proved in Section
4. Section 5 contains two results which show that $B_0(G)$ can be
as big as possible.

\bigskip
Acknowledgements. The subjects studied in this paper arose from
discussions with Yuri G. Prokhorov during the conferences
``Birational geometry and affine geometry" (April 23--27, 2012,
Moscow) and ``Essential dimensions and Cremona groups" (June
11--15, 2012, Nankai University). I thank Prokhorov and the
organizers of these two meetings, especially Ivan Cheltsov.

\begin{idef}{Standing notations.}
In discussing retract rationality, we always assume that the
ground field is infinite (see Definition \ref{d1.1}). For
emphasis, recall $k(G)=k(x_g:g\in G)^G$ defined in the second
paragraph of this section.

We denote by $\zeta_n$ a primitive $n$-th root of unity in some extension field of the ground field $k$.
When we write $\zeta_n\in k$ or $\fn{char}k \nmid n$,
it is understood that either $\fn{char}k=0$ or $\fn{char}k=p>0$ with $p\nmid n$.

All the groups in this paper are finite groups. Both $C_n$ and
$\bm{Z}/n\bm{Z}$ denote the same cyclic group of order $n$; we use
$\bm{Z}/n\bm{Z}$ when it appears as some cohomology group. The
exponent of a group $G$, $\exp (G)$, is defined as
$\exp(G)=\fn{lcm}\{\fn{ord}(g):g\in G\}$. If $G$ is a group,
$Z(G)$ and $[G,G]$ denote the center and the commutator subgroup
of $G$ respectively. If $g,h\in G$, then $[g,h]:=ghg^{-1}h^{-1}$,
and $G^{ab}$ denotes the quotient group $G/[G,G]$. We denote by
$\bm{F}_p$ the finite field with $p$ elements.
\end{idef}

A semidirect product group $G$ is denoted by $G=N\rtimes G_0$
where $N$ is a normal subgroup of $G$. If $\sigma \in N \subset G$
and $g\in G_0 \subset G$, we will write $\side{g}{\sigma}$ for
$g\sigma g^{-1}$.

\section{\boldmath Some properties of $B_0(G)$}

In this paper, if $N$ is a subgroup of $G$ and $M$ is a
$\bm{Z}[G]$-module, we will write $\fn{res}:H^q(G,M)\to H^q(N,M)$
and $\fn{cor}:H^q(N,M)\to H^q(G,M)$ for the restriction map and
the corestriction map of cohomology groups \cite[Chapter 3]{Br}.
If $\alpha$ is a $q$-cocycle of $G$ with coefficients in $M$, we
denote by $[\alpha]$ the cohomology class of $\alpha$, i.e.\
$[\alpha]\in H^q(G,M)$.

\begin{lemma} \label{l2.1}
Let $G$ be a finite group, $M$ be a $\bm{Z}[G]$-module.
If $\{N_1,N_2,\ldots,N_t\}$ is a collection of subgroups of $G$ satisfying that $\gcd\{[G:N_i]:1\le i\le t\}=1$,
then $\fn{res}:H^q(G,M)\to \bigoplus_{1\le i\le t} H^q(N_i,M)$ is injective.
\end{lemma}

\begin{proof}
The proof is the same as that of the collection $\{P_j:1\le j\le s\}$ where $P_j$ is a $p_j$-Sylow subgroup of $G$
with $|G|=p_1^{e_1}p_2^{e_2}\cdots p_s^{e_s}$.
We sketch the proof.
Let $\alpha$ be a $q$-cocycle such that $[\alpha]\in\fn{Ker} (\fn{res})$.
It follows that $\fn{res}_G^{N_i} (\alpha)$ is a coboundary for $1\le i\le t$.
Consider $\fn{cor}:H^q(N_i,M)\to H^q(G,M)$.
We find that $\fn{cor}_{N_i}^G\circ \fn{res}_G^{N_i}(\alpha)$ is a coboundary.
But $\fn{cor}_{N_i}^G\circ \fn{res}_G^{N_i}$ is the multiplication by $[G:N_i]$.
Apply the assumption that $\gcd \{[G:N_i]:1\le i\le t\}=1$. Done.
\end{proof}

\begin{defn} \label{d2.2}
Let $M$ be a $\bm{Z}[G]$-module with trivial $G$-actions.
For any $q\ge 1$, define
\[
B_M^q(G)=\bigcap_A \fn{Ker}\{\fn{res}_G^A:H^q(G,M)\to H^q(A,M)\}
\]
where $A$ runs over all the bicyclic subgroups of $G$. When
$M=\bm{Q}/\bm{Z}$, $B_M^2(G)$ is nothing but $B_0(G)$.
\end{defn}

\begin{lemma} \label{l2.3}
Let $G$ be a finite group, $M$ be a $\bm{Z}[G]$-module with trivial $G$-actions.

{\rm (1)} Let $N$ be a subgroup of $G$.
For any $q\ge 1$, the map $\fn{res}:H^q(G,M)\to H^q(N,M)$ induces a natural map $\fn{res}:B_M^q(G)\to B_M^q(N)$.

{\rm (2)} Let $N$ be a normal subgroup of $G$. The action of $G$
on $H^q(N,M)$ via the conjugation map of $G$ on $N$ induces a
well-defined action on $B_M^q(N)$. Moreover, the image of the map
$\fn{res}:B_M^q(G)\to B_M^q(N)$ defined in (1) is contained in
$B_M^q(N)^G$. In particular, the map $\fn{res}:B_M^q(G)\to
B_M^q(N)^G$ is well-defined.
\end{lemma}

\begin{proof}
(1) is obvious. We will prove (2).

Let $\alpha$ be a $q$-cocycle representing $[\alpha]\in B_M^q(N)$.

If $g\in G$, the $q$-cocycle $\side{g}{\alpha}$ is defined as:
$\side{g}{\alpha}(\tau_1,\tau_2,\ldots,\tau_q)=\alpha(g^{-1}\tau_1g,g^{-1}\tau_2g,g^{-1}\tau_qg)$ for any $\tau_1,\tau_2,\ldots,\tau_q\in N$,
because $G$ acts trivially on $M$.
It is easy to see that $\fn{res}_N^A(\side{g}{\alpha})$ is a coboundary for any bicyclic subgroup $A$ of $N$ (since $[\alpha]\in B_M^q(N)$).
Thus $B_0(N)$ is invariant under the action of $G$.

We will show that $\fn{Image}\{\fn{res}:B_M^q(G)\to B_M^q(N)\}\subset B_M^q(N)^G$.

Let $\alpha$ be a $q$-cocycle with $[\alpha]\in B_M^q(G)$. Then
$[\fn{res}_G^N(\alpha)]\in B_M^q(N)$. We will show that
$\side{g}{(\fn{res}_G^N(\alpha))}$ is cohomologous to
$\fn{res}_G^N(\alpha)$ for any $g\in G$.

For any $\tau_1,\ldots,\tau_q\in N$,
$\side{g}{(\fn{res}_G^N(\alpha))}
(\tau_1,\ldots,\tau_q)=\alpha(g^{-1}\tau_1g,g^{-1}\tau_2g,\ldots,g^{-1}\tau_qg)$.

On the other hand, consider the action of $G$ on $H^q(G,M)$ by the
conjugation map. Since $\side{g}{\alpha}$ is cohomologous to
$\alpha$ \cite[page 116]{Se}, it follows that there is a
$(q-1)$-cochain $\beta$ such that
$\side{g}{\alpha}=\alpha+\delta(\beta)$ where $\delta$ is the
differential map. Thus
$\alpha(g^{-1}\sigma_1g,g^{-1}\sigma_2g,\ldots$,
$g^{-1}\sigma_qg)=\alpha(\sigma_1,\sigma_2,\ldots,\sigma_q)+\delta(\beta)(\sigma_1,\ldots,\sigma_q)$
for any $\sigma_1,\ldots,\sigma_q\in G$. Substitute $\sigma_i$ by
$\tau_i\in N$, we get
$\alpha(g^{-1}\tau_1g,\ldots,g^{-1}\tau_qg)=\alpha(\tau_1,\ldots,\tau_q)+\delta(\beta)(\tau_1,\ldots,\tau_q)$.
Hence the result.
\end{proof}

Before proving Theorem \ref{t1.4},
we recall the definition of the exterior product $G\wedge G$ of a non-abelian group $G$.

\begin{defn}[{\cite[page 588; BL, page 316]{Mi}}] \label{d2.4}
Let $G$ be a group. The exterior product $G\wedge G$ is a group
defined by $G\wedge G=\langle g\wedge h: g,h\in G\rangle$ with
relations $g\wedge g=1$, $(g_1g_2)\wedge h=(\side{g_1}{g_2}\wedge
\side{g_1}{h})(g_1\wedge h)$, $g\wedge(h_1h_2)=(g\wedge
h_1)(\side{h_1}{g}\wedge \side{h_1}{h_2})$ for any
$g,h,g_1,g_2,h_1,h_2\in G$ where $\side{g}{h}:=ghg^{-1}$ for any
$g,h\in G$.
\end{defn}

\begin{theorem}[{Miller \cite[pages 592--593; BL, page 316]{Mi}}] \label{t2.5}
Let $G$ be a group, $[G,G]$ its commutator subgroup.

{\rm (1)} Define a group homomorphism $\gamma: G\wedge G\to [G,G]$ by $\gamma(g\wedge h)=ghg^{-1}h^{-1}\in [G,G]$ for any $g,h\in G$.
Define $M(G)=\fn{Ker}(\gamma)$.
Then $M(G)\simeq H_2(G,\bm{Z})$.

{\rm (2)} If $G=G_1\times G_2$, define $\Phi: M(G_1)\times M(G_2)\times (G_1\otimes G_2)\to M(G)$ by $\Phi(g_1\wedge g_2)=g_1\wedge g_2\in M(G)$ for any $g_1,g_2\in G_1$,
$\Phi(h_1\wedge h_2)=h_1 \wedge h_2\in M(G)$ for any $h_1,h_2\in G_2$,
$\Phi(g\otimes h)=g\wedge h \in M(G)$ for any $g\in G_1$, $h\in G_2$ where $G_1\otimes G_2:=(G_1^{ab} \otimes_{\bm{Z}} G_2^{ab})$,
the usual tensor product of abelian groups.
Then $\Phi$ is a group isomorphism.
\end{theorem}

\begin{proof}[Proof of Theorem \ref{t1.4}] ~

Step 1.
For a group $G$, the association of $G$ to the exact sequence $1\to M(G)\to G\wedge G\to [G,G]\to 1$ is a functor,
i.e.\ if $G\to H$ is a group homomorphism,
then there are group homomorphisms $f_1$, $f_2$, $f_3$ such that the following diagram commutes
\[
\xymatrix{1 \ar[r] & M(G) \ar[r] \ar[d]^{f_1} & G\wedge G \ar[r] \ar[d]^{f_2} & [G,G] \ar[r] \ar[d]^{f_3} & 1 \\
1 \ar[r] & M(H) \ar[r] & H \wedge H \ar[r] & [H,H] \ar[r] & 1}
\]

\bigskip
Step 2. Note that $\fn{Hom}(H_2(G,\bm{Z}),\bm{Q}/\bm{Z})\simeq
H^2(G,\bm{Q}/\bm{Z})$ by the universal coefficient theorem
\cite[page 8]{Br}. Since $\fn{Hom}(-,\bm{Q}/\bm{Z})$ is an
anti-equivalence on the category of finitely generated abelian
groups, it follows that $B_0(G)\simeq
\fn{Hom}(M(G)/M_0(G),\bm{Q}/\bm{Z})$ where $M_0(G)=\langle x\wedge
y\in G\wedge G: xy=yx\rangle$. See \cite[page 471 and 475]{Bo}
also.

\bigskip
Step 3. By Step 2, it suffices to show that $M(G_1)/M_0(G_1)
\times M(G_2)/M_0(G_2) \to M(G)/M_0(G)$ is a group isomorphism.

Apply Theorem \ref{t2.5}. Since $\Phi(G_1\otimes G_2)=G_1\wedge
G_2 \subset M_0(G)$, it follows that the induced morphism $M(G_1)
\times M(G_2) \to M(G)/\langle G_1\wedge G_2 \rangle$ is an
isomorphism. This isomorphism sends $M_0(G_1)$ and $M_0(G_2)$ to
$M_0(G)$. This induces an epimorphism $\Phi^* : M(G_1)/M_0(G_1)
\times M(G_2)/M_0(G_2) \to M(G)/M_0(G)$.

\bigskip
Step 4. We claim that $\Phi^*$ is injective.

It suffices to show that $M_0(G)$ is generated by $G_1\wedge G_2$
and the images of $M_0(G_1)$, $M_0(G_2)$.

For any $x,y\in G=G_1\times G_2$ with $xy=yx$, write $x=g_1g_2$,
$y=h_1h_2$ where $g_1,h_1\in G_1$, $g_2,h_2\in G_2$. Since all the
$g_i$, $h_j$ commute with each other. Apply the defining relations
in Definition \ref{d2.4}. It is easy to see that $x\wedge
y=x\wedge (h_1h_2)=(x\wedge h_1)(\side{h_1}{x}\wedge
\side{h_1}{h_2})=((g_1g_2)\wedge h_1))(\side{h_1}{(g_1g_2)}\wedge
h_2) =\cdots=(g_2\wedge h_1)(g_1\wedge h_1)(g_2\wedge
h_2)(g_1\wedge h_2)$. Note that $g_2\wedge h_1,g_1\wedge h_2 \in
G_1\wedge G_2$, $g_1\wedge h_1\in \fn{Image} \{M_0(G_1)\to
M_0(G)\}$, $g_2\wedge h_2\in \fn{Image}\{M_0(G_2)\to M_0(G)\}$.
Done.
\end{proof}

For the proof of Theorem \ref{t1.5},
note that Theorem \ref{t1.5} is a special case of the following theorem.

\begin{theorem} \label{t2.6}
Let $G=N\rtimes G_0$ be a finite group, $M$ be a
$\bm{Z}[G]$-module with trivial $G$-actions. If
$\gcd\{|N|,|G_0|\}=1$, then the restriction map $\fn{res}:
B_M^q(G)\to B_M^q(N)^{G_0} \times B_M^q(G_0)$ is an isomorphism
where $q\ge 1$ and $B_M^q(G)$ is the group defined in Definition
\ref{d2.2}.
\end{theorem}

\begin{proof}

Step 1. By Lemma \ref{l2.3}, the map $\fn{res}: B_M^q(G)\to
B_M^q(N)^{G_0} \times B_M^q(G_0)$ is well-defined because
$B_M^q(N)^G=B_M^q(N)^{G_0}$.

Write $|N|=n$, $|G_0|=m$.
Then $n\cdot B_M^q(N)=m\cdot B_M^q(G_0)=0$.
We will show that the corestriction map $\fn{cor}_N^G:H^q(N,M)\to H^q(G,M)$ induces a well-defined map $\fn{cor}:B_M^q(N)^{G_0} \to B_M^q(G)$.

\medskip
For any $q$-cocycle $\gamma$ such that $[\gamma]\in B_M^q(N)^{G_0}$,
we will show that $[\fn{cor}_N^G(\gamma)]\in B_M^q(G)$,
i.e.\ $\fn{res}_G^A\circ \fn{cor}_N^G(\gamma)$ is a coboundary for any bicyclic subgroup $A$ of $G$.

Write $A=\langle x,y \rangle \subset G$.
Define $A_1=\langle x^m, y^m\rangle$, $A_2=\langle x^n,y^n\rangle$.
Then $A=A_1\times A_2$ and $A_1$ (resp.\ $A_2$) is an abelian group of exponent dividing $n$ (resp.\ $m$).
Since $\gcd\{n,m\}=1$ and $N\lhd G$,
it follows that $A_1\subset N$.

By Lemma \ref{l2.1}, the map $\fn{res}:H^q(A,M)\to
H^q(A_1,M)\times H^q(A_2,M)$ is injective. To show that
$\fn{res}_G^A\circ \fn{cor}_N^G(\gamma)$ is a coboundary, it
suffices to show that both $\fn{res}_A^{A_1}\circ
\fn{res}_G^A\circ \fn{cor}_N^G(\gamma)$ and $\fn{res}_A^{A_2}\circ
\fn{res}_G^A\circ \fn{cor}_N^G(\gamma)$ are coboundaries.

\medskip
Note that $\fn{res}_A^{A_2}\circ \fn{res}_G^A\circ \fn{cor}_N^G(\gamma)=\fn{res}_G^{A_2}\circ \fn{cor}_N^G(\gamma)$ and $[\fn{res}_G^{A_2}\circ \fn{cor}_N^G(\gamma)]\in H^q(A_2,M)$ is of order dividing $|A_2|$.
On the other hand, $[\gamma]\in H^q(N,M)$ is of order dividing $|N|$.
Since $\gcd\{|N|,|A_2|\}=1$,
we find $[\fn{res}_G^{A_2}\circ \fn{cor}_N^G(\gamma)]=0$.

Now $\fn{res}_A^{A_1}\circ \fn{res}_G^A \circ
\fn{cor}_N^G(\gamma)=\fn{res}_N^{A_1}\circ \fn{res}_G^N\circ
\fn{cor}_N^G(\gamma)$ because $A_1\subset N$. By \cite[page
81]{Br}, $\fn{res}_G^N \circ \fn{cor}_N^G(\gamma)=\sum_{g\in G_0}
\side{g}{\gamma}$. Since $[\gamma]\in H^q(N,M)^{G_0}$, we find
that $\fn{res}_G^N\circ \fn{cor}_N^G(\gamma)=m\gamma$. Thus
$\fn{res}_A^{A_1}\circ \fn{res}_G^A\circ
\fn{cor}_N^G(\gamma)=m\cdot \fn{res}_N^{A_1}(\gamma)$. Since
$\gamma\in B_M^q(N)$, it follows that $\fn{res}_N^{A_1}(\gamma)$
is a coboundary. Done.

\bigskip
Step 2.
Since $\gcd\{n,m\}=1$,
we may write $B_M^q(G)=B_1\times B_2$ where $B_1=\{[\gamma]\in B_M^q(G): n[\gamma]=0\}$,
$B_2=\{[\gamma]\in B_M^q(G):m[\gamma]=0\}$.

Consider the restriction map $\fn{res}:B_M^q(G)\to
B_M^q(N)^{G_0}\times B_M^q(G_0)$. Note that $\fn{res}(B_1)\subset
B_M^q(N)^{G_0}$, $\fn{res}(B_2)\subset B_M^q(G_0)$. Hence it
remains to show that both $\fn{res}:B_1\to B_M^q(N)^{G_0}$ and
$\fn{res}:B_2\to B_M^q(G_0)$ are isomorphisms.

\bigskip
Step 3.
Consider the map $\fn{res}:B_1\to B_M^q(N)^{G_0}$ first.

Note that the following diagram commutes
\[
\xymatrix{B_1\ar[r]^{\fn{res}} \ar[d]_{f_1} & B_M^q(N)^{G_0} \ar[d]^{f_2} \\
H^q(G,M) \ar[r]_{\fn{res}_G^N} & H^q(N,M) }
\]
where $f_1$ and $f_2$ are inclusion maps. Since the composite map
$H^q(G,M) \xrightarrow{\fn{res}_G^N} H^q(N,M)$
$\xrightarrow{\fn{cor}_N^G} H^q(G,M)$ is the multiplication by $m$
map \cite[page 82]{Br}, it follows that $\fn{cor}_N^G\circ
\fn{res}_G^N: B_1\to H^q(G,M)$ is injective. In particular,
$\fn{res}_G^N:B_1\to H^q(N,M)$ and $\fn{res}:B_1\to B_M^q
(N)^{G_0}$ are injective maps.

\medskip
By Step 1, the corestriction map $\fn{cor}:B_M^q(N)^{G_0} \to B_M^q(G)$ is well-defined.
Since $n\cdot B_M^q(N)^{G_0}=0$,
we find that $\fn{cor}:B_M^q(N)^{G_0}\to B_1$ is also well-defined.
Consider the commutative diagram
\[
\xymatrix{B_M^q(N)^{G_0} \ar[r]^{\fn{cor}} \ar[d]_{f_2} & B_1 \ar[d]^{f_1} \\
H^q(N,M) \ar[r]_{\fn{cor}_N^G} & H^q(G,M) }
\]

The composite map $H^q(N,M)\xrightarrow{\fn{cor}_N^G} H^q(G,M)\xrightarrow{\fn{res}_G^N} H^q(N,M)$ is the norm map by \cite[page 81]{Br},
i.e.\ $\fn{res}_G^N\circ \fn{cor}_N^G([\gamma])=\sum_{g\in G_0} [\side{g}{\gamma}]$ for any $[\gamma]\in H^q(N,M)$.
Since $[\side{g}{\gamma}]=[\gamma]$ for any $[\gamma]\in B_M^q(N)^{G_0}$,
it follows that $\fn{res}_G^N\circ \fn{cor}_N^G: B_M^q(N)^{G_0}\to B_M^q(N)^{G_0}$ is the multiplication by $m$ map.
For any $[\gamma]\in B_M^q(N)^{G_0}$,
find $[\gamma']\in B_M^q(N)^{G_0}$ such that $[\gamma]=m[\gamma']=\fn{res}\circ \fn{cor}([\gamma'])=\fn{res}(\fn{cor}([\gamma']))$.
Hence $\fn{res}:B_1\to B_M^q(N)^{G_0}$ is surjective.

\bigskip
Step 4. Consider the map $\fn{res}:B_2\to B_M^q(G_0)$.

The injectivity of $\fn{res}:B_2\to B_M^q(G_0)$ follows also from
the similar fact that the composite map
$H^q(G,M)\xrightarrow{\fn{res}_G^{G_0}} H^q(G_0,M)
\xrightarrow{\fn{cor}_{G_0}^G} H^q(G,M)$ is the multiplication by
$n$ map. The details are omitted.

For the surjectivity of $\fn{res}:B_2\to B_M^q(G_0)$, we will show
that the inflation map $\inf_{G_0}^G: H^q(G_0,M)\to H^q(G,M)$
induces a well-defined map $\inf:B_M^q(G_0)\to B_M^q(G)$; here we
identify $G_0$ with $G/N$.

Note that $\fn{inf}_{G_0}^G: H^q(G_0,M)\to H^q(G,M)$ is defined
as: for any $q$-cocycle $\gamma$ with $[\gamma]\in H^q(G_0,M)$,
for any $x_1=g_1h_1$, $x_2=g_2h_2$, $\ldots$, $x_q=g_qh_q$ where
$g_1,g_2,\ldots,g_q\in N$, $h_1,h_2,\ldots,h_g\in G_0$, the
cocycle $\inf(\gamma)$ is defined by
$\inf(\gamma)(x_1,\ldots,x_q)=\gamma(h_1,h_2,\ldots,h_q)$.

Assume that $\inf: B_M^q(G_0)\to B_M^q(G)$ is well-defined, whose
proof will be given in Step 5. Then $\inf: B_M^q(G_0)\to B_2$ is
also well-defined. Consider the following commutative diagram
\[
\xymatrix{B_M^q(G_0) \ar[r]^{\inf} \ar[d]_{f_3} & B_2 \ar[r]^{\fn{res}} \ar[d]^{f_4} & B_M^q(G_0) \ar[d]^{f_3} \\
H^q(G_0,M) \ar[r]_{\inf_{G_0}^G} & H^q(G,M) \ar[r]_{\fn{res}_G^{G_0}} & H^q(G_0,M) }
\]
where $f_3$, $f_4$ are the inclusion maps.

Since the composite of the bottom row is the identity map,
so is that of the top row.
Hence $\fn{res}: B_2\to B_M^q(G_0)$ is surjective.

\bigskip
Step 5.
It remains to show that $\inf_{G_0}^G:H^q(G_0,M)\to H^q(G,M)$ induces a well-defined map $\inf:B_M^q(G_0)\to B_M^q(G)$.

For any bicyclic subgroup $A$ of $G$,
for any $[\gamma]\in B_M^q(G_0)$,
we will show that $\fn{res}_G^A\circ \inf_{G_0}^G ([\gamma])=0$.

Write $A=\langle x,y\rangle$, $A_1=\langle x^m,y^m\rangle$, $A_2=\langle x^n,y^n\rangle$.
Applying Lemma \ref{l2.1} as in Step 1,
we will show that the images of $[\gamma]$ in $H^q(G_0,M)\xrightarrow{\inf_{G_0}^G} H^q(G,M) \xrightarrow{\fn{res}_G^{A_1}} H^q(A_1,M)$ and
$H^q(G_0,M)\xrightarrow{\inf_{G_0}^G} H^q(G,M)\xrightarrow{\fn{res}_G^{A_2}} H^q(A_2,M)$ are zero.

Since $m\cdot H^q(G_0,M)=0=n\cdot H^q(A_1,M)$, the composite map
$H^q(G_0,M)\to H^q(G,M)\to H^q(A_1,M)$ is the zero map. This
solves the case of $A_1$.

\medskip
For the case of $H^q(A_2,M)$, it requires some work. Recall
$A=\langle x,y\rangle$ is bicyclic. Write $x=g_1h_1$, $y=g_2h_2$
where $g_1,g_2\in N$, $h_1,h_2\in G_0$. It is routine to verify
that $x^n=x_1h_1^n$, $y^n=y_1h_2^n$ for some $x_1,y_1\in N$, which
are irrelevant in our subsequent proof. Similarly, for any
non-negative integers $a$, $b$, we find
$(x^n)^a(y^n)^b=uh_1^{na}h_2^{nb}$ for some $u\in N$.

Now consider $\fn{res}_G^{A_2}\circ \inf_{G_0}^G(\gamma)$.
For any $(x^n)^{a_1} (y^n)^{b_1},\ldots,(x^n)^{a_q} (y^n)^{b_q}\in A_2=\langle x^n, y^n\rangle$ where $a_i$, $b_i$ are any non-negative integers,
we find
\begin{align*}
& \left(\fn{res}_G^{A_2}\circ \fn{inf}_{G_0}^G(\gamma)\right) \left((x^n)^{a_1} (y^n)^{b_1},\ldots,(x^n)^{a_q}(y^n)^{b_q}\right) \\
={} & \gamma(h_1^{na_1} h_2^{nb_1},\ldots,h_1^{na_q}h_2^{nb_q}) \\
={} & \fn{res}_{G_0}^{A_0} (\gamma) (h_1^{na_1} h_2^{nb_1}, \ldots, h_1^{na_q} h_2^{nb_q})
\end{align*}
where $A_0=\langle h_1^n,h_2^n \rangle \subset G_0$.

Since $\gamma\in B_M^q(G_0)$, we find that $\fn{res}_{G_0}^{A_0}(\gamma)$ is a coboundary.
\end{proof}

\begin{idef}{Remark.}
Tahara studies the Schur multiplier $M(G)$ of a finite group
$G=N\rtimes G_0$ \cite[page 33]{Ta,Kar}. In case
$\gcd\{|N|,|G_0|\}=1$, Tahara's formula becomes $M(G)\simeq
M(N)^{G_0}\times M(G_0)$, which resembles the formula in Theorem
\ref{t2.6}. Such a formula is valid for $H^q(G,M)$ with $q \ge 1$;
this fact was observed by Karpilovsky for $q=2$ \cite[page
35]{Kar}. We record it as follows.
\end{idef}

\begin{theorem} \label{t2.7}
Let $G=N\rtimes G_0$ be a finite group, $M$ be a
$\bm{Z}[G]$-module with trivial $G$-actions. If
$\gcd\{|N|,|G_0|\}=1$, then the restriction map
$\fn{res}:H^q(G,M)\to H^q(N,M)^{G_0}\times H^q(G_0,M)$ is an
isomorphism for all $q\ge 1$.
\end{theorem}

\begin{proof}
This result is implicit in the proof of Theorem \ref{t2.6}. The
details of the proof are omitted.

An alternative proof is to use the Hochschild-Serre spectral
sequence $E_2^{p,q}=H^p(G_0,$ $H^q(N,M)) \Rightarrow H^n(G, M)$.
Since $\gcd\{|N|,|G_0|\}=1$, it follows that $E_2^{p,q}=0$ (if $p
\ge 1$ and $q \ge 1$) and $E_2^{0,1} \to E_2^{2,0}$ is the zero
map. It follows that $E_2^{p,q}=E_3^{p,q}= \cdots
=E_{\infty}^{p,q}$. Hence we get a short exact sequence $0 \to
E_{\infty}^{n,0} \to H^n(G,M) \to E_{\infty}^{0,n} \to 0$, which
splits because $\gcd\{|N|,|G_0|\}=1$. Note that
$E_{\infty}^{n,0}=H^n(G_0, M^N)=H^n(G_0, M)$ because $G$ acts
trivially on $M$. The map $E_{\infty}^{n,0}=H^n(G_0, M^N) \to
H^n(G,M)$ is the inflation map.
\end{proof}

\begin{idef}{Remark.}
Let $G=N\rtimes G_0$. Assume that $N$ is a perfect group (but
without assuming $\gcd\{|N|,|G_0|\}=1$), Kunyavskii proves that
$M(G)\to M(N)^{G_0}\times M(G_0)$ is an isomorphism \cite[page
212]{Ku}.
\end{idef}

\bigskip
We give an application of Theorem \ref{t1.5} (or Theorem
\ref{t2.6}) to Frobenius groups. Recall that a Frobenius group $G$
is a semidirect product $G=N\rtimes G_0$ satisfying that, if
$\sigma \in N\backslash \{1\}$ and $g\in G_0\backslash \{1\}$,
then $g\sigma g^{-1}\ne \sigma$. In particular,
$\gcd\{|N|,|G_0|\}=1$. The subgroup $N$ is called the Frobenius
kernel of $G$, the subgroup $G_0$ is a Frobenius complement of $G$
(see, for example, \cite[pages 181--182; Ka2]{Is}).

\begin{theorem} \label{t2.8}
Let $G=N\rtimes G_0$ be a finite Frobenius group with kernel $N$ and complement $G_0$.
Then the restriction map $\fn{res}:B_0(G)\to B_0(N)^{G_0}$ is an isomorphism.
\end{theorem}

\begin{proof}
A Frobenius complement is a $GZ$-group by \cite[Theorem 1.4,
Definition 1.5]{Ka2}. The classification of such $GZ$-groups are
given in \cite[Theorem 2.7, Theorem 2.8, Theorem 2.9]{Ka2}.
Applying \cite[Theorem 4.7, Theorem 4.10, Theorem 4.11]{Ka2}, we
find that $\bm{C}(G_0)$ is retract $\bm{C}$-rational. In
particular, $B_0(G_0)=0$. Now apply Theorem \ref{t1.5}.
\end{proof}

\section{The lattice method}

In this section we recall several preliminaries which will be used in the proof of Theorem \ref{t1.9}.

\begin{theorem}[{\cite[Theorem 1]{HK2}}] \label{t3.1}
Let $L$ be any field and $G$ be a finite group acting on $L(x_1,\ldots,x_m)$,
the rational function field of $m$ variables over the field $L$.
Suppose that

{\rm (i)} for any $\sigma \in G$, $\sigma(L)\subset L$,

{\rm (ii)} the restriction of the action of $G$ to $L$ is faithful;

{\rm (iii)} for any $\sigma\in G$,
\[
\begin{pmatrix} \sigma(x_1) \\ \vdots \\ \sigma(x_m) \end{pmatrix}
=A(\sigma) \begin{pmatrix} x_1 \\ \vdots \\ x_m \end{pmatrix}+B(\sigma)
\]
where $A(\sigma)\in GL_m(L)$ and $B(\sigma)$ is an $m\times 1$ matrix over $L$.

Then there is a matrix $(\alpha_{ij})_{1\le i,j \le m} \in
GL_m(L)$ and $\beta_i \in L$ for $1\le i\le m$ such that, by
defining
\[
z_i:=\sum_{1\le j\le m} \alpha_{ij}x_j+\beta_i \quad \text{for }
1\le i\le m,
\]
we have $L(x_1,\ldots,x_m)=L(z_1,\ldots,z_m)$ and
$\sigma(z_i)=z_i$ for all $\sigma\in G$, for all $1\le i\le m$.

Moreover, if $B(\sigma)=0$ for all $\sigma \in G$,
then we may choose $\beta_1=\beta_2=\cdots=\beta_m=0$ in the above definition of $z_1,\ldots,z_m$.
\end{theorem}

\begin{theorem}[{Saltman \cite[Proposition 3.6; Ka1, Lemma 3.4]{Sa3}}] \label{t3.2}
Let $k\subset L$ be infinite fields, $L(x_1,\ldots,x_m)$ be the
rational function field of $m$ variables over $L$. Then $L$ is
retract $k$-rational if and only if so is $L(x_1,\ldots,x_m)$.
\end{theorem}

\begin{defn} \label{d3.3}
Let $G$ be a finite group. A finitely generated $\bm{Z}[G]$-module
$M$ is called a $G$-lattice if it is a free abelian group when the
$G$-action is forgotten. A $G$-lattice $M$ is called a permutation
lattice if it can be written as $M=\bigoplus_{1\le i\le n}
\bm{Z}\cdot e_i$ such that, for all $\sigma\in G$ and for all
$1\le i\le n$, $\sigma\cdot e_i=e_j$ (where $j$ depends on
$\sigma$ and $i$). A $G$-lattice $M$ is called an invertible
lattice if there is some $G$-lattice $M'$ such that $M\oplus M'$
is a permutation lattice. A $G$-lattice $M$ is called a flabby
lattice (or, a flasque lattice) if $H^{-1}(H,M)=0$ for all
subgroups $H$ of $G$ (note that $H^{-1}(G, M)$ is the Tate
cohomology). In the category of flabby $G$-lattices, we introduce
a similarity relation: Two $G$-lattices $M_1$ and $M_2$ are
similar if and only if $M_1\oplus P\simeq M_2\oplus Q$ for some
permutation lattices $P$ and $Q$. The equivalence class containing
$M$ is denoted by $[M]$.
\end{defn}

The theory of flabby lattices was developed by Swan, Endo and Miyata, Voskresenskii, Lenstra, Colliot-Th\'el\`ene and Sansuc, etc.
The reader can find an account of this theory in \cite{Sw1,Lo}.
Also see \cite{Sw4,Le}.

\begin{defn} \label{d3.4}
It is known that, for any $G$-lattice $M$, there is an exact
sequence of $\bm{Z}[G]$-modules $0\to M \to P\to E\to 0$ where $P$
is a permutation lattice and $E$ is a flabby lattice \cite[Lemma
8.5]{Sw1}. We define $[M]^{fl}=[E]$. Note that $[M]^{fl}$ is
well-defined. For, if $0\to M\to Q\to F\to 0$ is another
resolution where $Q$ is permutation and $F$ is flabby, then
$P\oplus F\simeq Q\oplus E$ \cite[Lemma 8.7]{Sw1}. Thus $[E]=[F]$.

In particular, $[M]^{fl}=0$ if and only if there is a exact
sequence of $G$-lattices $0\to M\to P\to Q\to 0$ where $P$ and $Q$
are permutation lattices. In case $[M]^{fl}=[E]$ for some
invertible lattice $E$, we will say that $[M]^{fl}$ is invertible
by abusing the notation.
\end{defn}

Note that $[M]^{fl}$ is denoted by $\rho(M)$ in \cite[page 33]{Sw1}.
We adopt the notation of Lorenz \cite[page 38]{Lo}.

\begin{defn} \label{d3.5}
Let $G$ be a finite group, $M$ be a $G$-lattice with
$M=\bigoplus_{1\le i\le n} \bm{Z}\cdot e_i$, and $K$ be a field.
We define $K(M)=K(x_1,\ldots,x_n)$ the rational function field of
$n$ variables over $K$ where $n=\fn{rank}_{\bm{Z}}M$. Suppose that
$G$ has an action on $K$ with $K^G=k$ (it is permissible that $G$
acts faithfully on $K$ or $G$ acts trivially on $K$). We define a
$G$-action on the field $K(M)$ by $\sigma (x_j)=\prod_{1\le i\le
n} x_i^{a_{ij}}$ if $\sigma\cdot e_j =\sum_{1\le i\le n}a_{ij}
e_i$ in $M$; the action of $G$ on the base field $K$ is the
prescribed $G$-action on $K$. The fixed field of $K(M)$ under this
$G$-action is denoted by $K(M)^G$.
\end{defn}

In other words, if $G$ acts trivially on $K$, i.e.\ $k=K$, the
$G$-action on $K(M)=k(x_1,\ldots,x_n)$ is called a purely monomial
action of $G$ in \cite{HK1}. If $G$ acts faithfully on $K$, i.e.\
$G\simeq \fn{Gal}(K/k)$, then $K(M)^G$ is the function field of an
algebraic torus over $k$, split by $K$ with character module $M$
\cite{Vo}. In general, the action of $G$ on $K(M)$ is called a
purely quasi-monomial action in \cite[Definition 1.1]{HKK}.

\begin{theorem} \label{t3.6}
Let $K/k$ be a finite Galois extension of fields, $G=\fn{Gal}(K/k)$, and $M$ be a $G$-lattice.
\leftmargini=8mm
\begin{enumerate}
\item[{\rm (1)}] {\rm (Lenstra \cite[page 304]{Le})}
$K(M)^G$ is stably $k$-rational if and only if $[M]^{fl}=0$.
\item[{\rm (2)}] {\rm (Saltman \cite[page 189]{Sa3})}
$K(M)^G$ is retract $k$-rational if and only $[M]^{fl}$ is invertible.
\end{enumerate}
\end{theorem}

\begin{theorem}[{Kang \cite[Theorem 5.4]{Ka1}}] \label{t3.7}
Let $G$ be a finite group, $M$ be a faithful $G$-lattice, i.e.\
for any $\sigma \in G\backslash \{1\}$ $\sigma\cdot v\ne v$ for
some $v\in M$. Let $G$ act on $k(M)$ by $k$-automorphisms. Then
$k(G)$ is retract $k$-rational if and only if so is $k(M)^G$ for
any faithful $G$-lattice $M$ (resp.\ for some faithful $G$-lattice
$M$) satisfying that $[M]^{fl}$ is invertible.
\end{theorem}

\begin{theorem} \label{t3.8}
Let $G$ be a finite group such that all the Sylow subgroups of $G$ are cyclic.
Let $M$ be a $G$-lattice.
Then $[M]^{fl}$ is invertible.
Moreover, if $k$ is an infinite field with $\zeta_e\in k$ where $e=\exp (G)$,
then both $k(G)$ and $k(M)$ are retract $k$-rational.
\end{theorem}

\begin{proof}
$[M]^{fl}$ is invertible by Endo-Miyata's Theorem \cite[Theorem
1.5; Sw2, Theorem 3.4; Ka1, Theorem 3.5]{EM}.

By \cite[Theorem 5.16, page 160]{Is}, both $[G,G]$ and $G/[G,G]$ are cyclic groups.
By \cite[Theorem 5.10]{Ka1}, $k(G)$ is retract $k$-rational.

It remains to show that $k(M)^G$ is also retract $k$-rational.
Define $H=\{\tau\in G:\tau\cdot v=v$ for any $v\in M\}$.
Then $M$ is a faithful $G/H$-lattice.

It is clear that all Sylow subgroups of $G/H$ are cyclic. Hence
$k(G/H)$ is retract $k$-rational by the above arguments. Regard
$M$ as a faithful $G/H$-lattice. Then $[M]^{fl}$ is invertible by
Endo-Miyata's Theorem again. Thus we may apply Theorem \ref{t3.7}
to show that $k(M)^{G/H}$ is retract $k$-rational. But
$k(M)^G=\{k(M)^H\}^{G/H}=k(M)^{G/H}$.
\end{proof}

\begin{theorem} \label{t3.9}
Let $G$ be a finite group, $0\to M\to F\to B\to 0$ be an exact
sequence of $\bm{Z}[G]$-modules such that $F$ is a finitely
generated projective $\bm{Z}[G]$-module. Suppose that, for some
positive integer $e$, $e\cdot B=0$. If $\gcd\{|G|,e\}=1$, then $M$
is a projective $\bm{Z}[G]$-module.
\end{theorem}

\begin{proof}
This theorem is essentially a restatement of a result of Swan \cite[Proposition 7.1]{Sw3}.

Write $n=|G|$.
Find integers $s$ and $t$ such that $sn+te=1$.

Define a map $\phi:B\to B$ by $\phi(b)=s\cdot b$ for any $b\in B$.
Note that $\phi$ is a homomorphism of abelian groups, it is not a
$\bm{Z}[G]$-morphism.

For any $\sigma \in G$, define $\side{\sigma}{\phi}: B\to B$ by
$\side{\sigma}{\phi}(b):=\sigma(\phi(\sigma^{-1}\cdot b))$. Note
that $\side{\sigma}{\phi}(b)=\sigma(\phi(\sigma^{-1}\cdot
b))=\sigma (s\cdot(\sigma^{-1}\cdot b))=s\cdot b$. Hence
$(\sum_{\sigma \in G} \side{\sigma}{\phi})(b)=ns\cdot
b=(1-te)\cdot b=b$.

Thus the $\bm{Z}[G]$-module $B$ is weakly projective in the sense
of \cite[Proposition 1.1, page 233]{CE}. Hence it is
cohomologically trivial by \cite[Proposition 2.2, page 236]{CE}.
Applying \cite[Theorem 4.12]{Ri}, we find that $B$ has homological
dimension $\le 1$. But $0\to M\to F\to B\to 0$ is a resolution of
$B$. Thus $M$ is a projective module.
\end{proof}

\section{Proof of Theorem \ref{t1.9}}

The following theorem is a special case of Theorem \ref{t1.6}. We
include it here for two reasons. First, Saltman proves Theorem
\ref{t1.6} by constructing a generic Galois $G$-extension while
our proof of Theorem \ref{t4.2} is a direct proof showing $k(G)$
is retract $k$-rational. Second, Saltman's Theorem is the
prototype of Jambor's Theorem. A good understanding of (a variant
of) Saltman's Theorem, i.e. Theorem \ref{t4.2}, is helpful to
clarify the main idea of the proof of Theorem \ref{t1.9}.

\begin{theorem} \label{t4.2}
Let $G=N\rtimes G_0$ be a finite group where $N$ is an abelian
normal subgroup of exponent $e$. Let $k$ be an infinite field with
$\zeta_e\in k$. If $\gcd\{|N|,|G_0|\}=1$ and $k(G_0)$ is retract
$k$-rational, then $k(G)$ is also retract $k$-rational.
\end{theorem}

\begin{proof}
Step 1. Recall that $k(G)=k(x(g'):g'\in G)^G$. We will work on the
$k$-vector space $X=\bigoplus_{g'\in G} k\cdot x(g')$ with a left
$G$-action defined by $h'\cdot x(g')=x(h'g')$ for any $g',h'\in
G$.

There is also a right $G$-action on $X$ defined by $x(g')\cdot
h'=x(g'h')$ for any $g',h'\in G$. Note that, if $g_1,g_2,g_3\in
G$, then $(g_1\cdot x(g_2))\cdot g_3=g_1\cdot (x(g_2)\cdot g_3)$.
Thus the group algebra $\Lambda:=\{\sum_i a_i\cdot g'_i:a_i\in k,
g'_i\in G\}$ acts on $X$ from the left and from the right. The
reason to introduce the right action of $G$ on $X$ is that it
facilitate to write elements of $X$ in a succinct way in the
sequel.

\bigskip
Step 2. Consider the $k$-vector subspace $Y=\bigoplus_{\tau\in N}
k\cdot x(\tau)$. The group $N$ acts on $Y$ and this group action
can be diagonalized simultaneously because $N$ is abelian of
exponent $e$ and $\zeta_e\in k$.

Explicitly, define $N^*=\fn{Hom}_{\bm{Z}}(N,\langle
\zeta_e\rangle)$ which is a $\bm{Z}[G_0]$-module defined as: for
any $g\in G_0$, any $\chi\in N^*$, $\side{g}{\chi} \in N^*$ and
the character $\side{g}{\chi}$ is defined by
$\side{g}{\chi}(\tau):=\chi(g^{-1}\tau g)$ for any $\tau\in N$.
Now, for any $\chi\in N^*$, define $w(\chi)=\sum_{\tau\in
N}\chi(\tau^{-1})\cdot x(\tau)\in Y$. Clearly
$\sigma(w(\chi))=\chi(\sigma)\cdot w(\chi)$ for all $\sigma\in N$.
Thus $Y=\bigoplus_{\chi \in N^*} k\cdot w(\chi)$.

It follows that $X=\bigoplus_{\chi,g} k\cdot (w(\chi)\cdot g)$
where $\chi$ runs all characters in $N^*$ and $g$ runs all
elements in $G_0$. Thus $k(x(g'):g'\in G)=k(w(\chi)\cdot g:\chi
\in N^*, g \in G_0)$.

\medskip
Note that $\sigma\cdot (w(\chi)\cdot g)=\chi(\sigma)\cdot
(w(\chi)\cdot g)$ for any $\sigma\in N$, any $g\in G_0$. Also note
that $w(\chi)\cdot g=\sum_{\tau\in N} \chi(\tau^{-1})x(\tau
g)=\sum_{\tau\in N}\chi(\tau^{-1})x(g\cdot g^{-1}\tau g)=g\cdot
w(\side{g^{-1}}{\chi})$.

\bigskip
Step 3. Define $P:=\bigoplus_{\chi,g} \bm{Z}\cdot (w(\chi)\cdot
g)$ where $\chi\in N^*$ and $g\in G_0$. It follows that
$P=\bigoplus_{\chi,g}\bm{Z}\cdot (w(\chi)\cdot
g)=\bigoplus_{\chi,g} \bm{Z}\cdot(g\cdot
w(\side{g^{-1}}{\chi}))=\bigoplus_{\psi,g}\bm{Z}\cdot(g\cdot
w(\psi))$ where $g\in G_0$ and $\chi\in N^*$ (resp.\ $\psi\in
N^*$).

\medskip
For any $h \in G_0$, define $h \cdot (g\cdot w(\psi))=hg \cdot
w(\psi)$. Thus $P$ becomes a $G_0$-lattice. For a fixed $\psi \in
N^*$, the $\bm{Z}[G_0]$-module $\bigoplus_{g\in G_0} \bm{Z}\cdot
(g\cdot w(\psi))$ is isomorphic to $\bm{Z}[G_0]$. Hence $P$ is a
free $\bm{Z}[G_0]$-module of rank equal to $|N|$.

\medskip
Define a map $\Phi:P\to N^*$ by $\Phi(w(\chi)\cdot g)=\chi$. Note
that, for any $h \in G_0$, $\Phi (h \cdot (w(\chi)\cdot
g))=\Phi(w(\side{h}\chi)\cdot hg)=\side{h}\chi=h \cdot
\Phi(w(\chi)\cdot g)$. It follows that $\Phi$ is a morphism of
$\bm{Z}[G_0]$-modules.

Define $M=\fn{Ker} (\Phi)$. Then we have an exact sequence of
$\bm{Z}[G_0]$-modules $0\to M\to P\to N^*\to 0$.

\bigskip
Step 4. Since $k(P)=k(w(\chi)\cdot g:\chi \in N^*, g \in G_0)$ and
$\sigma\cdot (w(\chi)\cdot g)=\chi(\sigma)\cdot (w(\chi)\cdot g)$
for any $\sigma\in N$, it is easy to see that $M$ is the set of
all ``monomials" (in the variables $w(\chi)\cdot g$) which are
left fixed by the action of $N$. We conclude that
$k(M)=k(P)^N=k(x(g'):g'\in G)^N$. It follows that
$k(M)^{G_0}=\{k(P)^N\}^{G_0}=k(G)$

The above deduction that $k(M)^{G_0}=k(G)$ was generalized by
Saltman in Theorem \ref{t4.1}. It is useful in other rationality
problems.

\bigskip
Step 5. Since $\gcd\{|N|,|G_0|\}=1$ and $P$ is a free
$\bm{Z}[G_0]$-module by Step 3, we may apply Theorem \ref{t3.9} to
the sequence $0\to M\to P\to N^*\to 0$. Thus $M$ is a projective
$\bm{Z}[G_0]$-module also. Hence it is an invertible
$G_0$-lattice; in particular, $[M]^{fl}$ is invertible.

Since $k(G_0)$ is retract $k$-rational by assumption,
we may apply Theorem \ref{t3.7} to $k(M)^{G_0}$.
It follows that $k(M)^{G_0}=k(G)$ is retract $k$-rational.
\end{proof}

We record the theorem of Saltman mentioned in Step 4 of the above
proof.

\begin{theorem}[{\cite[Theorem 3.1; Be, Lemma 2.7]{Sa4}}] \label{t4.1}
Let $G=N\rtimes G_0$ where $N$ is an abelian normal subgroup of
exponent $e$. Define $N^*=\fn{Hom}_{\bm{Z}}(N,\bm{Z}/e \bm{Z})$
and regard $N^*$ as a $\bm{Z}[G_0]$-module where $G_0$ acts on
$\bm{Z}/e \bm{Z}$ trivially. Suppose that $0\to M\to P\to N^*\to
0$ is an exact sequence of $\bm{Z}[G_0]$-modules such that $P$ is
a permutation $G_0$-lattice and $M$ is a faithful $G_0$-lattice.
If $k$ is an infinite field with $\zeta_e\in k$ and $G_0$ acts on
$k(M)$ by $k$-automorphisms, then $k(M)^{G_0}$ is stably
isomorphic to $k(G)$ over $k$.
\end{theorem}

\bigskip
\begin{proof}[Proof of Theorem \ref{t1.9}] ~

Step 1. As in Step 1 of the proof of Theorem \ref{t4.2}, define a
$k$-vector space $X=\bigoplus_{g'\in G}k\cdot x(g')$ with a left
$G$-action. It follows that $k(G)=k(x(g')): g\in G)^G$. Moreover,
define a right action of $G$ on $X$ by $x(g')\cdot h'=x(g'h')$ for
any $g',h'\in G$. Thus the group algebra $\Lambda:=\{\sum_i
a_i\cdot g'_i:a_i\in k,g'_i\in G\}$ acts on $X$ from the left and
from the right.

Similarly, for any prime number $p\mid |N|$, define a $k$-vector
space $\tilde{X}_p=\bigoplus_{g'\in G}k\cdot x_p(g')$ with a left
$G$-action and a right $G$-action.

Choose any prime number $p_1$ with $p_1\mid |N|$. Embed $X$ into
$\bigoplus_{p\mid |N|} \tilde{X}_p$ by sending $x(g')$ to
$x_{p_1}(g')$ for any $g' \in G$.

Apply Theorem \ref{t3.1}. We find that $k(x_p(g'):g'\in G, p\mid
|N|)^G$ is rational over $k(x(g'):g'\in G)^G=k(G)$. By Theorem
\ref{t3.2}, $k(x_p(g'):g'\in G, p\mid |N|)^G$ is retract
$k$-rational if and only if so is $k(G)$. It suffices to show that
$k(x_p(g'):g'\in G, p\mid |N|)^G$ is retract $k$-rational.

\bigskip
Step 2. Consider the representation $G_0\to GL(W_{\fn{reg}})$
where $W_{\fn{reg}}$ is the regular representation space of $G_0$.
Consider the composite map $G\to G_0\to GL(W_{\fn{reg}})$. Thus we
get a $k$-vector space $X_0=\bigoplus_{g\in G_0} k\cdot x_0(g)$
with a left $G$-action defined by $\tau\cdot x_0(g)=x_0(g)$,
$h\cdot x_0(g)=x_0(hg)$ for any $g,h\in G_0$, any $\tau\in N$.

Consider $k(x_p(g'),x_0(g):g'\in G, g\in G_0, p\mid |N|)$. Apply
Theorem \ref{t3.1} again. We find that $k(x_p(g'), x_0(g):g'\in
G,g\in G_0, p\mid |N|)^G$ is rational over $k(x_p(g'):g'\in G,
p\mid |N|)^G$. By Theorem \ref{t3.2}, $k(x_p(g'),x_0(g):g'\in
G,g\in G_0,p\mid |N|)^G$ is retract $k$-rational if and only if so
is $k(x_p(g'):g'\in G, p\mid |N|)^G$. It remains to show that
$k(x_p(g'),x_0(g):g'\in G,g\in G_0,p\mid |N|)^G$ is retract
$k$-rational.

\bigskip
Define $K=k(x_0(g):g\in G_0)$. Then $K(x_p(g'):g'\in G,p\mid
|N|)=k(x_p(g'),x_0(g):g'\in G,g\in G_0,p\mid |N|)$. We will show
that $K(x_p(g'):g'\in G,p\mid |N|)^G$ is retract rational over
$K^G\simeq k(G_0)$.

Assuming this, we find that $K(x_p(g'):g'\in G,p\mid |N|)^G$ is
retract $k$-rational by \cite[Theorem 4.2]{Ka1} because $k(G_0)$
is retract $k$-rational by assumption. This finishes the proof
that $k(G)$ is retract $k$-rational.

In summary, our purpose is to show that $K(x_p(g'):g'\in G,p\mid
|N|)^G$ is retract rational over $K^G$ where $K=k(x_0(g):g\in
G_0)$.

\bigskip
Step 3. Recall the notations in Definition \ref{d1.7} where
$N=\prod_{p\mid |N|} N_p$ and $H_p=\{g\in G_0: g\tau g^{-1}$
$=\tau$ for any $\tau\in N_p\}$. Define $N'_p$ to be the subgroup
of $N$ generated by Sylow subgroups of $N$ other than $N_p$. As
before, define $N^*=\fn{Hom}_{\bm{Z}}(N,\langle \zeta_e\rangle)$
and $N_p^*=\fn{Hom}_{\bm{Z}}(N_p,\langle \zeta_e\rangle)$ with the
natural structures of $\bm{Z}[G_0]$-modules.

We may write $N^{*}= \oplus_{p\mid |N|} N_p^{*}$ in the sense
that, for any $\chi \in N_p^{*}$ and any $\sigma \in N'_p$, we
require that $\chi(\sigma)=1$.

For any prime number $p\mid |N|$, define a subspace $Y_p$ of
$\tilde{X}_p$ by defining $Y_p:=\bigoplus_{\tau \in N_p} k\cdot
y(\tau)$ with $y(\tau)$ defined as $y(\tau):=\sum_\sigma
x_p(\tau\sigma)$ where $\sigma$ runs over elements in $N'_p$. It
is easy to verify that $\tau'\cdot y(\tau)=y(\tau'\tau)$, $\sigma
\cdot y(\tau)=y(\tau)$ for any $\tau,\tau'\in N_p$, any $\sigma
\in N'_p$.

Define a subspace $X_p$ of $\tilde{X}_p$ by
$X_p:=\bigoplus_{\tau,g} k\cdot (y(\tau)\cdot g)$ where $\tau\in
N_p$, $g\in G_0$.

\medskip
Note that $X_0\oplus \bigoplus_{p\mid |N|} X_p$ is a faithful
$G$-subspace (from the left) of $X_0\oplus \bigoplus_{p\mid |N|}
\tilde{X}_p$. Apply Theorem \ref{t3.1}. We find that
$K(x_p(g'):g'\in G,p\mid |N|)^G=K(y(\tau)\cdot g:\tau\in N_p$,
$g\in G_0$, $p\mid |N|)^G(u_1,\ldots,u_t)$ for some variables
$u_1,\ldots,u_t$ with $g'(u_i)=u_i$ for all $g'\in G$, for all
$1\le i\le t$. By Theorem \ref{t3.2}, $K(x_p(g'):g'\in G,p\mid
|N|)^G$ is retract rational over $K^G$ if and only if so is
$K(y(\tau)\cdot g: \tau\in N_p$, $g\in G_0$, $p\mid |N|)^G$ over
$K^G$.

In conclusion, our goal is to prove that $K(y(\tau)\cdot g:
\tau\in N_p$, $g\in G_0$, $p\mid |N|)^G$ is retract rational over
$K^G$.

\bigskip
Step 4. In Step 4 -- Step 6, $p$ denotes a prime divisor of $|N|$,
which is fixed in the discussion.

For any $\chi\in N_p^*$, define $w(\chi)=\sum_{\tau\in N_p}
\chi(\tau^{-1})\cdot y(\tau)\in Y_p$. It is not difficult to show
that $\tau\cdot w(\chi)=\chi(\tau)\cdot w(\chi)$ and $\sigma\cdot
w(\chi)=w(\chi)$ for any $\tau\in N_p$, for all $\sigma\in N'_p$.

\medskip
On the other hand, write $H_p=\{g_1,\ldots,g_s\}$ with $s=|H_p|$.
The group $H_p$ acts naturally on $K(x_p(g_i):1\le i\le s)$ and
the restriction of this action to $K$ is faithful. By Theorem
\ref{t3.1}, there is a matrix $A:=(\alpha_{ij})\in GL_s(K)$ such
that $K(x_p(g_i):1\le i\le s)=K(\tilde{u}_i:1\le i\le s)$ where
$\tilde{u}_i=\sum_{1\le l\le s}\alpha_{il}\cdot x_p(g_l)$ and
$g\cdot \tilde{u}_i=\tilde{u}_i$ for all $g\in H_p$, for all $1\le
i\le s$.

\medskip
For $1\le i\le s$ and $\chi\in N_p$, define $u_i(\chi):=\sum_{1\le
l\le s} \alpha_{il}\cdot (w(\chi)\cdot g_l)\in \bigoplus_{\tau,g}
K\cdot (y(\tau)\cdot g)$. Note that $u_i(\chi) \in K\otimes_k X_p$
; in fact, $\bigoplus_{i,\chi} K\cdot u_i(\chi)=\bigoplus_{i,\tau}
K\cdot (y(\tau)\cdot g_i)$ where $1\le i\le s$, $\chi\in N_p^*$
and $\tau\in N_p$.

Since $N$ acts trivially on $K$, it is easy to verify that, if
$\tau\in N_p$, $\sigma\in N'_p$ and $\chi\in N_p^*$, then
$\tau\cdot u_i(\chi)=\chi(\tau)\cdot u_i(\chi), \, \sigma\cdot
u_i(\chi)=u_i(\chi)$.

In the next step, we will show that, if $g\in H_p$, $\chi\in
N_p^*$ and $1\le i\le s$, then $g\cdot u_i(\chi)=u_i(\chi)$.

\bigskip
Step 5. First we note that, if $\tau\in N_p$ and $g\in G_0$, then
$y(\tau)\cdot g=g\cdot y(g^{-1}\tau g)$; similarly, if $\chi \in
N_p^*$ and $1\le l\le s$, then $w(\chi)\cdot g_l=g_l\cdot w(\chi)$
because each $g_l\in H_p$.

\medskip
If $g\in H_p$, $\chi\in N_p^*$ and $1\le i\le s$, then we have
\begin{align*}
g\cdot u_i(\chi) &= g\cdot \sum_{1\le l\le s} \alpha_{il}\cdot (w(\chi)\cdot g_l) \\
&= \sum_{1\le l\le s} g(\alpha_{il})\cdot (g(g_l\cdot w(\chi))) \\
&= \sum_{1\le l\le s} g(\alpha_{il})\cdot ((g\cdot g_l)\cdot \sum_{\tau\in N_p} \chi (\tau^{-1})\cdot y(\tau)) \\
&= \sum_{1\le l\le s} g(\alpha_{il})\cdot ((g\cdot g_l)\cdot
\sum_{\tau,\sigma} \chi(\tau^{-1})\cdot x_p(\tau\sigma))
\end{align*}
where $\tau\in N_p$, $\sigma\in N'_p$.

The right-hand-side of the above identity becomes
\begin{align*}
& \sum_{1\le l\le s} g(\alpha_{il})\cdot (x_p(g\cdot g_l)\cdot \sum_{\tau,\sigma} \chi(\tau^{-1}) \cdot \tau\sigma) \\
={}& (\sum_{1\le l\le s} g(\alpha_{il})\cdot x_p(g\cdot g_l))\cdot
\sum_{\tau,\sigma} \chi(\tau^{-1})\cdot \tau\sigma
\end{align*}
where the right factor is an element in the group algebra
$\Lambda:=\{\sum_i a_i\cdot g'_i:a_i\in k,$ $g'_i\in G\}$ (thus we
write the term as $\tau\sigma$, not as $x_p(\tau\sigma)$).

Continue the computation of the right-hand-side of the above
identity. We get
\begin{align*}
& g(\tilde{u}_i)\cdot (\sum_{\tau,\sigma} \chi(\tau^{-1})\cdot \tau\sigma) \\
={}& \tilde{u}_i\cdot (\sum_{\tau,\sigma} \chi(\tau^{-1})\cdot
\tau\sigma).
\end{align*}

Note that the right-hand-side is independent of the element $g\in
H_p$. We conclude that $g\cdot u_i(\chi)=u_i(\chi)$ for all $g\in
H_p$.

\bigskip
Step 6. Write $G_0=\bigcup_{1\le j\le t} h_j H_p$ where
$|G_0|=t\cdot |H_p|=st$. Note that $s$ and $t$ depend on the prime
factor $p$; for simplicity, we will use the notation $s$ and $t$
instead of $s(p)$ and $t(p)$.

Define $u_{ij}(\chi)=h_j(u_i(\chi))$ for $1\le i\le s$, $1\le j\le
t$, $\chi\in N_p^*$.

Clearly $\bigoplus_{i,j,\chi} K\cdot u_{ij}(\chi)$ is the induced
representation space of $\bigoplus_{i,\chi} K\cdot
u_i(\chi)=\bigoplus_{i,\tau}K\cdot(y(\tau)\cdot g_i)$ in $K
\otimes X_p$. Counting the dimensions, we find that
$\{u_{ij}(\chi):1\le i\le s$, $1\le j\le t$, $\chi\in N_p^*\}$ is
a basis of $K \otimes X_p$. Thus $K(y(\tau)\cdot g: \tau\in
N_p,g\in G_0)=K(u_{ij}(\chi):1\le i\le s$, $1\le j\le t$, $\chi\in
N_p^*)$.

It is not difficult to check the action of $G_0$ on
$u_{ij}(\chi)$. It is given as follows: For any $h\in G_0$, if
$h\cdot h_j=h_{j'}\cdot g$ for some $1\le j'\le t$, some $g\in
H_p$, then $h\cdot u_{ij}(\chi)=h(h_j(u_i(\chi))=(h_{j'}\cdot
g)(u_i(\chi))=u_{ij'}(\chi)$.

\bigskip
Step 7. Define a $G_0$-lattice $F:=\bigoplus_{i,j,\chi,p}
\bm{Z}\cdot u_{ij}(\chi)$ where $1\le i\le s$, $1\le j\le t$,
$\chi\in N_p^*$ and $p\mid |N|$. For each $p\mid |N|$, define
$F_p:=\bigoplus_{i,j,\chi}\bm{Z}\cdot u_{ij}(\chi)$ where $1\le
i\le s$, $1\le j\le t$ and $\chi\in N_p^*$. Thus $F=\oplus_{p\mid
|N|}F_p$.

By the last remark of Step 6, $F_p$, $F$ are permutation
$G_0$-lattices. Moreover, $K(P)=K(u_{ij}(\chi):1\le i\le s$, $1\le
j\le t$, $\chi\in N_p^*$, $p\mid |N|)$ $=K(y(\tau)\cdot g: \tau\in
N_p$, $g\in G_0$, $p\mid |N|)$.

The action of $N$ on $u_{ij}(\chi)$ is given as follows. If
$\tau\in N_p$ and $\chi\in N_p^*$, then
\begin{align*}
\tau\cdot(u_{ij}(\chi)) &= \tau\cdot (h_j(u_i(\chi)))=(h_j\cdot h_j^{-1}\tau h_j)(u_i(\chi)) \\
&= \chi(h_j^{-1}\tau h_j)(h_j\cdot (u_i(\chi)))=(\side{h_j}{\chi})(\tau)\cdot u_{ij}(\chi).
\end{align*}

Similarly, if $\sigma\in N'_p$ and $\chi\in N_p^*$, then $\sigma(u_{ij}(\chi))=u_{ij}(\chi)$.

\bigskip
Step 8. The remaining proof is similar to Step 3 -- Step 5 of the
proof of Theorem \ref{t4.2}. Define a map $\Phi:F\to N^*$ by
$\Phi(u_{ij}(\chi))=\side{h_j}{\chi}$ if $\chi\in N_p^*$. Define
$M=\fn{Ker}(\Phi)$. Let $\Phi_p$ be the restriction of $\Phi$ to
$N_p^*$, i.e. $\Phi_p:F_p\to N_p^* \subset N^*$ with
$\Phi_p(u_{ij}(\chi))=\side{h_j}{\chi}$ (remember the isomorphism
$N^{*}= \oplus_{p\mid |N|} N_p^{*}$ in Step 3). Define
$M_p=\fn{Ker}(\Phi_p)$. Clearly $\Phi$ is a morphism of
$\bm{Z}[G_0]$-modules.

\medskip
Since $H_p$ acts trivially on $F_p$ by Step 6 and $H_p$ also acts
trivially on $N_p$, the exact sequence of $\bm{Z}[G_0]$-modules
$0\to M_p\to F_p\to N_p^*\to 0$ can be regarded as an exact
sequence of $\bm{Z}[G_0/H_p]$-modules.

Now use the assumption that either $p\nmid [G_0:H_p]$ or all the
Sylow subgroups of $G_0/H_p$ are cyclic. We find that $M_p$ is a
projective module over $\bm{Z}[G_0/H_p]$ (when we apply Theorem
\ref{t3.9}) or $[M_p]^{fl}$ is an invertible lattice over
$\bm{Z}[G_0/H_p]$ (when we apply Theorem \ref{t3.8}). In either
case, $[M_p]^{fl}$ is an invertible $G_0/H_p$-lattice. Thus it is
an invertible $G_0$-lattice.

\bigskip
Step 9. Remember that $N$ acts trivially on $K$. Thus
$K(F)^N=K(M)$ by the same arguments as in Step 4 of the proof of
Theorem \ref{t4.2}.

Since $\Phi(F_p) \subset N_p^{*}$ and $N^{*}= \oplus_{p\mid |N|}
N_p^{*}$ is a direct product, it follows that $M=\bigoplus_{p\mid
|N|} M_p$.

On the other hand, all these $M_p$ satisfy the condition that
$[M_p]^{fl}$ is an invertible $G_0$-lattice. Hence $[M]^{fl}$ is
also an invertible $G_0$-lattice.

Note that $G_0$ acts faithfully on $K$. Applying Theorem
\ref{t3.6}, we find that $K(M)^{G_0}$ is retract rational over
$K^{G_0} \simeq k(G_0)$. Because $K(M)^{G_0}=K(y(\tau)\cdot g:
g\in G_0$, $\tau\in N_p$, $p\mid |N|)^G$, we reach the goal stated
at the end of Step 3.
\end{proof}

Here is an application of Theorem \ref{t1.9}.

\begin{theorem} \label{t4.3}
Let $G=N\rtimes G_0$ be a finite group where $N$ is an abelian
normal subgroup and all the Sylow subgroups of $G_0$ are cyclic
groups. If $k$ is an infinite field with $\zeta_e\in k$ where
$e=\fn{lcm}\{\exp(N),\exp(G_0)\}$, then $k(G)$ is retract
$k$-rational.
\end{theorem}

\begin{proof}
By Theorem \ref{t3.8},
$k(G_0)$ is retract $k$-rational.
Apply Theorem \ref{t1.9}.
\end{proof}

\begin{idef}{Remark.}
Barge shows that, if $N$ is abelian and $G_0$ is bicyclic, then
$B_0(N\rtimes G_0)=0$  \cite[Theorem 3]{Ba}. We don't know, in the
same situation, whether $\bm{C}(G)$ is retract $\bm{C}$-rational.
However, if $G_0$ has a $p$-Sylow subgroup which is not bicyclic,
then there is a group $G=N\rtimes G_0$ such that $N$ is abelian
and $\bm{C}(G)$ is not retract rational. See \cite[page 234; Sa5,
page 543]{Sa4}; also see the proof of Theorem 2.6 of \cite{Ba}.
\end{idef}

\section{\boldmath $B_0(G)$ can be very big}

First we recall a construction of Schur covering groups of abelian
$p$-groups.

\begin{theorem}[{\cite[page 90]{Kar}}] \label{t5.1}
Let $p$ be a prime number, $n_1\ge n_2\ge \cdots \ge n_t\ge 1$.
Define an abelian $p$-group $\Gamma:=C_{p^{n_1}}\times
C_{p^{n_2}}\times \cdots \times C_{p^{n_t}}$. A Schur covering
group $\widetilde{G}$ of $\Gamma$ may be defined as
$\widetilde{G}=\langle \sigma_1,\ldots,\sigma_t\rangle$ with
$Z(\widetilde{G})=[\widetilde{G},\widetilde{G}]$ and with
relations
$\sigma_1^{p^{n_1}}=\sigma_2^{p^{n_2}}=\cdots=\sigma_t^{p^nt}=1$,
$[\sigma_i,\sigma_j]^{p^{n_j}}=1$ for $1\le i<j \le t$.
\end{theorem}

\begin{theorem} \label{t5.2}
Let $p$ be any prime number,
$n\ge 1$ be a positive integer.
Let $\widetilde{G}=\langle\sigma_1$, $\ldots$, $\sigma_{n+3}\rangle$ be the Schur covering group of the elementary abelian group $\Gamma$ of order $p^{n+3}$ defined in Theorem \ref{t5.1}.
Define a subgroup $H$ of order $p^n$ in $\widetilde{G}$ by $H=\langle [\sigma_1,\sigma_2][\sigma_3,\sigma_4]$, $[\sigma_1,\sigma_3][\sigma_4,\sigma_5]$, $[\sigma_1,\sigma_4][\sigma_5,\sigma_6]$,
$\ldots$, $[\sigma_1,\sigma_{n+1}][\sigma_{n+2},\sigma_{n+3}]\rangle$.
Define $G=\widetilde{G}/H$.
Then $B_0(G)$ contains a subgroup isomorphic to $(\bm{Z}/p\bm{Z})^n$.
\end{theorem}

\begin{proof}
The proof is similar to that of Theorem 2.3 in \cite{Hkku}.

\bigskip
Step 1. Let $Z(\widetilde{G})$ be the center of $\widetilde{G}$.
Define $N=Z(\widetilde{G})/H$. $N$ is an elementary abelian group
of order $p^{\binom{n+3}{2}-n}$. Consider the Hochschild-Serre
5-term exact sequence,
\begin{align*}
0 &\to H^1(G/N,\bm{Q}/\bm{Z}) \to H^1(G,\bm{Q}/\bm{Z}) \to H^1(N,\bm{Q}/\bm{Z})^G \\
&\xrightarrow{\psi'} H^2(G/N,\bm{Q}/\bm{Z}) \xrightarrow{\psi} H^2(G,\bm{Q}/\bm{Z})
\end{align*}
where $\psi'$ is the transgression map and $\psi$ is the inflation map.

We claim that the image of $\psi:H^2(G/N,\bm{Q}/\bm{Z})\to
H^2(G,\bm{Q}/\bm{Z})$ is isomorphic to an elementary abelian group
of order $p^n$.

\medskip
Since $G/N\simeq C_p^{n+3}$, it follows that $H^2(G/N,\bm{Q}/\bm{Z})$ is the elementary abelian group of order $p^{\binom{n+3}{2}}$ by \cite[page 37]{Kar}.

Since $N\simeq C_p^{\binom{n+3}{2}-n}$, $H^1(N,\bm{Q}/\bm{Z})$ is the elementary abelian group of order $p^{\binom{n+3}{2}-n}$.
Since $N=Z(G)$, it follows that $G$ acts trivially on $H^1(N,\bm{Q}/\bm{Z})$,
i.e.\ $H^1(N,\bm{Q}/\bm{Z})=H^1(N,\bm{Q}/\bm{Z})^G$.

On the other hand, $N=[G,G]$.
Hence $H^1(G/N,\bm{Q}/\bm{Z})\to H^1(G,\bm{Q}/\bm{Z})$ is an isomorphism.
Thus $\psi'$ is an injection.
It follows that the image of the restriction map $\psi:H^2(G/N,\bm{Q}/\bm{Z})\to H^2(G,\bm{Q}/\bm{Z})$ is a group of order $p^{\binom{n+3}{2}-\left[\binom{n+3}{2}-n\right]}=p^n$.

\bigskip
Step 2. We will show that the image of
$\psi:H^2(G/N,\bm{Q}/\bm{Z})\to H^2(G,\bm{Q}/\bm{Z})$ is contained
in $B_0(G)$. Assume it. The proof of this theorem is finished.

For any bicyclic subgroup $A=\langle x,y\rangle$ of $G$, we will
show that the composite map
$H^2(G/N,\bm{Q}/\bm{Z})\xrightarrow{\psi}
H^2(G,\bm{Q}/\bm{Z})\xrightarrow{\fn{res}} H^2(A,\bm{Q}/\bm{Z})$
is the zero map.

Consider the following commutative diagram
\[
\begin{array}{c}
H^2(G/N,\bm{Q}/\bm{Z}) \xrightarrow{\psi} H^2(G,\bm{Q}/\bm{Z}) \xrightarrow{\fn{res}} H^2 (A,\bm{Q}/\bm{Z}) \\[4pt]
{}^{\psi_0} \bigg\downarrow \hspace*{5.5cm} \bigg\uparrow {}^{\psi_1} \\[-2pt]
H^2(AN/N,\bm{Q}/\bm{Z}) \stackrel{\widetilde{\psi}}{\simeq} H^2(A/A\cap N,\bm{Q}/\bm{Z})
\end{array}
\]
where $\psi_0$ is the restriction map, $\psi_1$ is the inflation
map, $\widetilde{\psi}$ is the natural isomorphism.

We claim that $AN/N$ is a cyclic group.
Assuming this claim,
we find $H^2(AN/N$, $\bm{Q}/\bm{Z})=0$ by \cite[page 37]{Kar}.
Thus $\fn{res}\circ \psi=\psi_1\circ \widetilde{\psi}\circ \psi_0=0$.

It remains to show that $AN/N$ is cyclic.

\bigskip
Step 3. For any bicyclic group $A=\langle x,y\rangle$ of $G$, we
will show that $A N/N$ is cyclic. The proof is very similar to
Step 3 and Step 4 in the proof of Lemma 2.2 of \cite{Hkku}.

Since $G/N \simeq C_p^{n+3}$, we regard $G/N$ as a vector space
over the finite field $\bm{F}_p$. Let $\bar{x}$ and $\bar{y}$ be
the image of $x$ and $y$ in $G/N$ respectively. Thus
$\bar{x}=\prod_i \bar{\sigma}_i^{\lambda_i}$, $\bar{y}=\prod_i
\bar{\sigma}_i^{\mu_i}$ where $1\le i\le n+3$ and $0\le
\lambda_i,\mu_i\le p-1$ (we adopt the multiplicative notation for
$G/N$).

We will perform row operations on the matrix
\[
\begin{pmatrix}
\lambda_1 & \lambda_2 & \cdots &\lambda_{n+3} \\
\mu_1 & \mu_2 & \cdots & \mu_{n+3}
\end{pmatrix}.
\]

In each row operation, the basis
$\{\bar{\sigma_1},\ldots,\bar{\sigma}_{n+3}\}$ will not be
changed. But we may replace the generators $x$, $y$. For example,
if we replace $\mu_1$ by $\mu_1-a\lambda_1$ by a row operation,
then we replace of generator $x$, $y$ by $x$, $yx^{-a}$.

If $A N/N$ is not cyclic, there will exist $i,j$ with $1\le i<
j\le n+3$ satisfying that $A=\langle x,y\rangle \subset G$ and
$\bar{x}=\prod_{i\le s\le n+3} \bar{\sigma}_s^{a_s}$,
$\bar{y}=\prod_{j\le t\le n+3} \bar{\sigma}_t^{b_t}$ where $0\le
a_s,b_t\le p-1$, with the following property
\begin{equation}
a_i=b_j=1 \quad\text{and}\quad a_j=0. \label{eq1}
\end{equation}

It follows that $x=\prod_{i\le s\le n+3} \sigma_s^{a_s}u_1$,
$y=\prod_{j\le t\le n+3} \sigma_t^{b_t}u_2$ for some $u_1,u_2\in
N$.

Since $N=Z(G)$, the identity $xy=yx$ is equivalent to the
following identity
\begin{equation}
[\prod_{i\le s\le n+3} \sigma_s^{a_s},\prod_{j\le t\le n+3}
\sigma_t^{b_t}]=1. \label{eq2}
\end{equation}

\medskip
Because $Z(G)=[G,G]$, we find $[x,yz]=[x,y][x,z]$,
$[xy,z]=[x,z][y,z]$ for any $x,y,z\in G$.

From \eqref{eq2}, we obtain
\begin{equation}
\prod_{j\le t\le n+3} [\sigma_i,\sigma_t]^{a_ib_t}\cdot \prod_{i+1\le s<t \le n+3} [\sigma_s,\sigma_t]^{a_sb_t-a_tb_s}=1 \label{eq3}
\end{equation}
with the convention $b_{i+1}=b_{i+2}=\cdots=b_{j-1}=0$.

Suppose that $i\ge 2$. Note that $N\simeq
[\widetilde{G},\widetilde{G}]/H$ may be regarded as a vector space
over the finite field $\bm{F}_p$ with basis
$\{[\sigma_1,\sigma_{n+2}],[\sigma_1,\sigma_{n+3}],[\sigma_s,\sigma_t]:
2\le s<t\le n+3\}$.

It follows that $a_ib_j=a_ib_{j+1}=\cdots=a_ib_{n+3}=0$. Hence
$b_j=b_{j+1}=\cdots=b_{n+3}=0$. But this is impossible because
$b_j=1$ by \eqref{eq1}. From now on, we assume that $i=1$.

\bigskip
Step 4. Assume that $i=1$. Since
$[\sigma_1,\sigma_s][\sigma_{s+1},\sigma_{s+2}]\in H$ for $2\le
s\le n+1$, we find that
$[\sigma_1,\sigma_s]=[\sigma_{s+1},\sigma_{s+2}]^{-1}$ in $G$.
Note that $b_{n+2}=b_{n+3}=0$. Thus \eqref{eq3} becomes
\begin{equation}
\prod_{j\le t\le n+1} [\sigma_{t+1},\sigma_{t+2}]^{-b_t} \cdot \prod_{2\le s<t\le n+3} [\sigma_s,\sigma_t]^{a_sb_t-a_tb_s}=1. \label{eq4}
\end{equation}

Consider the term $[\sigma_j,\sigma_t]$ for $t\ge j+2$.
We find that $a_jb_t-a_tb_j=0$.
Since $a_j=0$ and $b_j=1$ by \eqref{eq1},
it follows that $a_t=0$ for $t\ge j+2$.

In summary, we have shown that $a_j=a_{j+2}=a_{j+3}=\cdots=a_{n+3}=0$ and $b_1=b_2=\cdots=b_{j-1}=b_{n+2}=b_{n+3}=0$ with $a_1=b_j=1$.

\medskip
The identity \eqref{eq2} becomes
\begin{gather}
\prod_{j\le t\le n+1} [\sigma_1,\sigma_t]^{b_t}\cdot \prod_{2\le s
\le j-1 < j \le t\le n+1} [\sigma_s,\sigma_t]^{a_sb_t}
\cdot \prod_{j\le t\le n+1} [\sigma_{j+1},\sigma_t]^{a_{j+1}b_t}=1, \nonumber \\
\prod_{j\le t\le n+1} [\sigma_{t+1},\sigma_{t+2}]^{-b_t}\cdot
\prod_{2\le s \le j-1 < j \le t\le n+1}
[\sigma_s,\sigma_t]^{a_sb_t} \cdot \prod_{j\le t\le n+1}
[\sigma_{j+1},\sigma_t]^{a_{j+1}b_t}=1. \label{eq5}
\end{gather}

From \eqref{eq5}, we find that $a_2=a_3=\cdots=a_{j+1}=0$. Thus
$b_j=b_{j+1}=\cdots=b_{n+1}=0$. Again this is impossible because
$b_j=1$ by \eqref{eq1}.
\end{proof}

\begin{idef}{Remark.}
When $n=1$ in the above theorem, we obtain the group $G$ of order
$p^9$ exhibited in the papers of Saltman and Shafarevich
\cite[page 83; Sh, page 245]{Sa2}.
\end{idef}

\begin{lemma} \label{l5.3}
Let $p$ be a prime number, $n_1\ge n_2\ge \cdots \ge n_t$, $t\ge
2$, $\Gamma=\langle x_1\rangle\times \langle x_2\rangle \times
\cdots \times \langle x_t\rangle$ be a direct product of cyclic
groups such that $\langle x_i\rangle \simeq C_{p^{n_i}}$ for $1\le
i\le t$. For $1\le i<j\le t$, define a $2$-cocycle $f_{ij}:
\Gamma\times \Gamma \to \bm{Q}/\bm{Z}$ where
$f_{ij}(x^I,x^J)=\zeta_{n_j}^{-a_jb_i}$ where
$x^I:=x_1^{a_1}x_2^{a_2}\cdots x_t^{a_t}$,
$x^J:=x_1^{b_1}x_2^{b_2}\cdots x_t^{b_t}$ and we adopt the
multiplicative notation for $\bm{Q}/\bm{Z}$ $($i.e.\ regard
$\bm{Q}/\bm{Z}$ as the group of all roots of unity in
$\bm{C}^\times)$. Then $H^2(\Gamma,\bm{Q}/\bm{Z})$ is generated by
the cohomology classes of these $f_{ij}$ where $1\le i<j \le t$.
\end{lemma}

\begin{proof}
See \cite[page 37 and 91]{Kar}. Here is a simple way to understand
this lemma. Assume that $t=2$. Construct the Schur covering group
$\widetilde{G}$ of $\Gamma$ by Theorem \ref{t5.1}. The central
extension $\varepsilon:1\to \langle c\rangle \to \widetilde{G}\to
\Gamma \to 1$ where $c=[\sigma_1,\sigma_2]$ is of order $p^{n_2}$
gives the 2-cocycle $f_{12}$. Explicitly, define a section $\mu:
\Gamma\to \widetilde{G}$ by $u(x^I)=\sigma^I$ where
$\sigma^I=\sigma_1^{a_1}\sigma_2^{a_2}$ if
$x^I=x_1^{a_1}x_2^{a_2}$. The extension defines a 2-cocycle by
$u(x^I)\cdot u(x^J)=\varepsilon(x^I,x^J)\cdot u(x^{I+J})$. Using
the fact that
$\sigma_2^{a_2}\sigma_1^{b_1}=[\sigma_1,\sigma_2]^{-a_2b_1}\sigma_1^{b_1}\sigma_2^{a_2}$,
we can verify easily $\varepsilon(x^I,y^J)=f_{12}(x^I,x^J)$ for
all $I$, $J$.
\end{proof}

\begin{theorem} \label{t5.4}
Let $p$ be a prime number, $n$ be a positive integer. Let
$\widetilde{G}=\langle \sigma_1,\sigma_2,\sigma_3,\sigma_4\rangle$
be the Schur covering group of the group $\Gamma\simeq C_{p^n}^4$
defined in Theorem \ref{t5.1}. Let $H$ be the cyclic subgroup of
$\widetilde{G}$ generated by
$[\sigma_1,\sigma_2][\sigma_3,\sigma_4]$. Define
$G=\widetilde{G}/H$. Then $B_0(G)$ contains a subgroup isomorphic
to $\bm{Z}/p^n\bm{Z}$.
\end{theorem}

\begin{proof}
The strategy is similar to that of Theorem \ref{t5.2}. We use the
Hochschild-Serre 5-term exact sequence. As in the proof of Theorem
\ref{t5.2}, define $N=Z(\widetilde{G})/H\subset G$. It is not
difficult to show that the following is an exact sequence
\[
0\to H^1(N,\bm{Q}/\bm{Z})^G \xrightarrow{\psi'} H^2(G/N,\bm{Q}/\bm{Z})\xrightarrow{\psi} H^2(G,\bm{Q}/\bm{Z})
\]
and $H^1(N,\bm{Q}/\bm{Z})^G\simeq (\bm{Z}/p^n\bm{Z})^5$, $H^2(G/N,\bm{Q}/\bm{Z})\simeq (\bm{Z}/p^n\bm{Z})^6$.
Hence the image of $\psi$ is isomorphic to $\bm{Z}/p^n \bm{Z}$.

Then we claim that $\psi(H^2(G/N,\bm{Q}/\bm{Z}))\subset B_0(G)$,
i.e.\ the composite map $H^2(G/N,$
$\bm{Q}/\bm{Z})\xrightarrow{\psi}
H^2(G,\bm{Q}/\bm{Z})\xrightarrow{\fn{res}} H^2(A,\bm{Q}/\bm{Z})$
is the zero map for all bicyclic subgroups $A$ of $G$. Since
$H^2(G/N,\bm{Q}/\bm{Z})$ is generated by the cohomology classes of
$f_{ij}$ (where $1\le i<j\le 4$) constructed in Lemma \ref{l5.3},
it suffices to show that the image of each $f_{ij}$ under the
composite map becomes a coboundary.

In conclusion, it remains to show that, for $1\le i<j\le 4$,
$\fn{res}_G^A \circ \psi (f_{ij})$ is a coboundary for any bicyclic subgroup $A=\langle x,y \rangle \subset G$.
Remember that $\psi$ is the inflation map $\inf_{G/N}^G$.

\bigskip
Step 1. Denote by
$\bar{\sigma}_1,\bar{\sigma}_2,\bar{\sigma}_3,\bar{\sigma}_4$ the
images of $\sigma_1,\sigma_2,\sigma_3,\sigma_4\in G$ in $G/N$.
Write $\sigma^I$ for the element
$\sigma_1^{a_1}\sigma_2^{a_2}\sigma_3^{a_3}\sigma_4^{a_4}$ where
$I=(a_1,a_2,a_3,a_4)$ with $0\le a_i \le p^n-1$; similarly
$\bar{\sigma}^I=\bar{\sigma}_1^{a_1} \bar{\sigma}_2^{a_2}
\bar{\sigma}_3^{a_3} \bar{\sigma}_4^{a_4}$.

If $A=\langle x,y\rangle \subset G$ is a bicyclic subgroup,
then $x=\sigma^I u_1$, $y=\sigma^J u_2$ for some $I=(a_1,a_2,a_3,a_4)$, $J=(b_1,b_2,b_3,b_4)$ with $0\le a_i,b_j\le p^n-1$, for some $u_1,u_2\in N$.
Note that $xy=yx$ if and only if $[\sigma^I,\sigma^J]=1$.
Expanding the commutator $[\sigma^I, \sigma^J]$ and using the definition of $H$,
we find that $[\sigma^I,\sigma^J]=1$ if and only if
\begin{align}
& a_ib_j-a_jb_i \equiv 0 \, (mod \, p^n) \quad \mbox{for $0\le i<j\le 4$ with $(i,j)\ne (1,2), (3,4)$, and} \label{eq6} \\
& a_1b_2-a_2b_1 \equiv a_3b_4-a_4b_3 \, (mod \, p^n). \label{eq7}
\end{align}

On the other hand, we can evaluate the 2-cocycle $\fn{res}_G^A
\circ \psi (f_{ij})$ as follows. Without loss of generality, we
may assume that $A=\langle x\rangle \times \langle y\rangle$ is a
direct product. Thus every element in $A$ can be written as
$x^{c_1}y^{c_2}$ for some non-negative integers $c_1$, $c_2$
(unique up to the orders of $x$ and $y$). If $c_1$, $c_2$, $d_1$,
$d_2$ are non-negative integers, then
\begin{equation}
\fn{res}_G^A \circ \psi(f_{ij}) (x^{c_1}y^{c_2},x^{d_1}y^{d_2})=f_{ij}(\bar{\sigma}^{c_1I+c_2J},\bar{\sigma}^{d_1I+d_2J})=\zeta^{-(c_1a_j+c_2b_j)(d_1a_i+d_2b_i)} \label{eq8}
\end{equation}
where $\zeta=\zeta_{p^n}$.

\bigskip
Step 2. Define a set $S:=\{I=(a_1,a_2,a_3,a_4): 0\le a_i \le
p^n-1\}$. Define an order function $\fn{ord}(I)$ on $S$ by
$\fn{ord}(I)=\max\{c\in \bm{N} \cup \{0\}:p^c$ divides
$a_1,a_2,a_3,a_4\}$ if $I\ne (0,0,0,0)$; if $I=(0,0,0,0)$ we
define $\fn{ord}(I)=p^n$.

Return to the bicyclic group $A=\langle x,y\rangle$ with $x=\sigma^I u_1$, $y=\sigma^J u_2$ where $I=(a_1,a_2$, $a_3,a_4)$ and $J=(b_1,b_2,b_3,b_4)$, $u_1,u_2\in N$.

If $\fn{ord}(I)=n=\fn{ord}(J)$, then $\fn{res}_G^A \circ
\psi(f_{ij}) (x^{c_1}y^{c_2},x^{d_1}y^{d_2})=1$ for all $c_1$,
$c_2$, $d_1$, $d_2$ by \eqref{eq8}, i.e. $\fn{res}_G^A \circ
\psi(f_{ij})$ is a coboundary and we are done.

From now on, we assume that at least one of $\fn{ord}(I)$ and
$\fn{ord}(J)$ is $\le n-1$.

\bigskip
Step 3. Without loss of generality, we may assume that
$\fn{ord}(I)\le \fn{ord}(J)$ and $e=\fn{ord}(I)$ with $0 \le e \le
n-1$. Thus we may assume that $p^e|a_i$ for $1\le i\le 4$, and
$p^{e+1}\nmid a_1$ (the case when $p^{e+1}\nmid a_i$ for other
indices $i$ can be treated similarly).

If $e \ge n/2$, then $\fn{res}_G^A \circ \psi(f_{ij})
(x^{c_1}y^{c_2},x^{d_1}y^{d_2})=1$ again by \eqref{eq8}. Thus we
may assume that $e < n/2$ from now on.

We may also assume that $b_1=0$. Otherwise, replace the generator
$y$ by $y^{a'_1} x^{-b'_1}$ where $a_1=p^ea'_1$ and $b_1=p^eb'_1$.

By \eqref{eq6}, we find that $b_3$ and $b_4$ are divisible by
$p^{n-e}$, because $b_1=0$. By \eqref{eq7} $b_2$ is also divisible
by $p^{n-e}$.

By \eqref{eq8} $\fn{res}_G^A \circ \psi(f_{ij})
(x^{c_1}y^{c_2},x^{d_1}y^{d_2})= \zeta^{-c_1d_1a_ia_j}$, because
$e < n/2$.

\bigskip
Step 4. We will show that the $2$-cocycle evaluated at the end of
the last step is a $2$-coboundary.

Consider the cyclic group $A_1=\langle x \rangle \subset A \subset
G$. Since $H^2(A_1,\bm{Q}/\bm{Z})=0$ by \cite[page 37]{Kar}, it
follows that $H^2(G/N,\bm{Q}/\bm{Z})\xrightarrow{\psi}
H^2(G,\bm{Q}/\bm{Z}) \xrightarrow{\fn{res}}
H^2(A_1,\bm{Q}/\bm{Z})$ is the zero map. Thus $\fn{res}_G^{A_1}
\circ \psi(f_{ij}) (x^{c_1},x^{d_1})$ is a coboundary for any
non-negative integers $c_1,d_1$.

On the other hand, we find that $\fn{res}_G^{A_1} \circ
\psi(f_{ij}) (x^{c_1},x^{d_1})= \zeta^{-c_1d_1a_ia_j}$ also by
\eqref{eq8}. But this is a $2$-coboundary by the previous
paragraph. Thus there is a $1$-cochain $\alpha : A_1 \to
\bm{Q}/\bm{Z}$ such that $\delta (\alpha) (x^{c_1},
x^{d_1})=\zeta^{-c_1d_1a_ia_j}$ where $\delta$ is the differential
map from the group of $1$-cochains to that of $2$-cochains.

Define a $1$-cochain $\beta : A \to \bm{Q}/\bm{Z}$ by $\beta
(x^{c_1}y^{c_2})=\alpha (x^{c_1})$. It follows that $\delta
(\beta) (x^{c_1}y^{c_2},$ $x^{d_1}y^{d_2})=\zeta^{-c_1d_1a_ia_j}$.
Done.

 \end{proof}

\begin{idef}{Remark.}
In \cite{CHKK} and \cite{Hkku}, it is proved that, if $G$ is a
group of order $p^5$ where $p$ is an odd prime number (resp.\ a
group of order $2^6$), then $B_0(G)\ne 0$ if and only if $G$
belongs to the 10-th isoclinism family (resp.\ $G$ belongs to the
13-rd isoclinism family, i.e.\ $G=G(2^6,i)$ with $241\le i\le
245$). For these groups $G$ with $B_0(G)\ne 0$, it can be shown
that $B_0(G) \simeq \bm{Z}/p\bm{Z}$ if $G$ is of order $p^5$, and
$B_0(G)\simeq \bm{Z}/2\bm{Z}$ if $G$ is of order $2^6$.

The proof is similar to the proof in \cite[Theorem 5.4]{Hkku}. For
example, in the case when $G$ is a group of order $p^5$ where $p$
is an odd prime number, choose a normal subgroup $N = \langle f_1,
f_3, f_4, f_5 \rangle$ when $G$ is the group defined in the proof
of \cite[Theorem 2.3]{Hkku}. Apply the 7-term exact sequence of
Hochschild-Serre \cite[Section 5]{DHW,Hkku}. Note that $B_0(G)$ is
a non-zero subgroup of $H^2(G, \bm{Q}/\bm{Z})_1$ while $0 \to
H^2(G, \bm{Q}/\bm{Z})_1 \to H^1(G/N, H^1(N,\bm{Q}/\bm{Z}))$ is an
exact sequence. It is not difficult to show that $H^1(G/N,
H^1(N,\bm{Q}/\bm{Z})) \simeq \bm{Z}/p\bm{Z}$. Hence $B_0(G) \simeq
\bm{Z}/p\bm{Z}$.
\end{idef}

\newpage
\renewcommand{\refname}{\centering{References}}


\begin{thebibliography}{CTS2}


\bibitem[Ba]{Ba}
J. Barge, \textit{Cohomologie des groupes et corps d'invariants multiplicatifs},
Math. Ann. 283 (1989), 519--528.

\bibitem[Bo]{Bo}
F. A. Bogomolov, \textit{The Brauer group of quotient spaces by
linear group actions}, Math. USSR Izv. 30 (1988), 455--485.

\bibitem[BL]{BL}
R. Brown and J.-L. Loday, \textit{Van Kampen theorems for diagrams
of spaces}, Topology 26 (1987), 311--335.


\bibitem[Br]{Br}
K. S. Brown, \textit{Cohomology of groups}, GTM vol.87,
Springer-Verlag, New York, 1982.


\bibitem[CE]{CE}
H. Cartan and S. Eilenberg,
\textit{Homological algebra}, Princeton University Press, Princeton, 1956.

\bibitem[CHKK]{CHKK}
H. Chu, S.-J. Hu, M. Kang and B. E. Kunyavskii, \textit{Noether's
problem and the unramified Brauer group for groups of order 64},
International Math. Research Notices 12 (2010), 2329--2366.

\bibitem[De]{De}
F. R. DeMeyer, \textit{Generic polynomials}, J. Algebra 84 (1983),
441--448.

\bibitem[DHW]{DHW}
K. Dekimpe, M. Hartl and S. Wauters,
\textit{A seven-term exact sequence for the cohomology of a group extension},
arXiv: 1103.4052, to appear in ``J. Algebra".

\bibitem[EM]{EM}
S. Endo and T. Miyata, \textit{On a classification of the function
fields of algebraic tori}, Nagoya Math. J. 56 (1974) 85--104;
Corrigenda, ibid. 79 (1980) 187--190.


\bibitem[HK1]{HK1}
M. Hajja and M. Kang, \textit{Three-dimensional purely monomial
group actions}, J. Algebra 170 (1994), 850--860.

\bibitem[HK2]{HK2}
M. Hajja and M. Kang, \textit{Some actions of symmetric groups},
J. Algebra 177 (1995), 511--535.

\bibitem[HKK]{HKK}
A. Hoshi, M. Kang and H. Kitayama,
\textit{Quasi-monomial actions and some 4-dimensional rationality problem}, arXiv: 1201.1332.

\bibitem[HKKu]{Hkku}
A. Hoshi, M. Kang and B. E. Kunyavskii, \textit{Noether's problem
and unramified Brauer groups}, arXiv: 1202.5812, to appear in
``Asian J. Math.".

\bibitem[Is]{Is}
I. M. Isaacs, \textit{Finite group theory}, GSM. vol.92, Amer.
Math. Soc., Providence, 2008.

\bibitem[Ja]{Ja}
S. Jambor, \textit{Generic extensions and generic polynormials for semidirect products},
J. Algebra 322 (2009), 4040--4052.

\bibitem[Ka1]{Ka1}
M. Kang, \textit{Retract rational fields}, J. Algebra 349 (2012), 22--37.

\bibitem[Ka2]{Ka2}
M. Kang, \textit{Frobenius groups and retract rationality}, arXiv: 1204.1796.

\bibitem[Kar]{Kar}
G. Karpilovsky, \textit{The Schur multiplier}, Clarendon Press, Oxford, 1987.

\bibitem[Ku]{Ku}
B. E. Kunyavskii, \textit{The Bogomolov multiplier of finite
simple groups}, in ``Cohomological and geometric approaches to
rationality problems", edited by F. Bogomolov and Y. Tschinkel,
Progress in Math. vol. 282, Birkh\"auser, Boston, 2010.


\bibitem[Le]{Le}
H. W. Lenstra, Jr., \textit{Rational functions invariant under a
finite abelian group}, Invent. Math. 25 (1974), 299--325.

\bibitem[Lo]{Lo}
M. Lorenz,
\textit{Multiplicative invariant theory},
Encyclo. Math. Sci. vol. 135, Springer-Verlag, 2005, Berlin.

\bibitem[Mi]{Mi}
C. Miller, \textit{The second cohomology of a group}, Proc. Amer.
Math. Soc. 3 (1952), 588--595.


\bibitem[MT]{MT}
Y. I. Manin and M. A. Tsfasman, \textit{Rational varieties:
algebra, geometry and arithmetric}, Russian Math. Survey 41
(1986), 51--116.


\bibitem[Ri]{Ri}
D. S. Rim, \textit{Modules over finite groups}, Ann. Math. 69 (1959), 700--712.

\bibitem[Sa1]{Sa1}
D. J. Saltman, \textit{Generic Galois extensions and problems in
field theory}, Advances in Math. 43 (1982), 250--283.

\bibitem[Sa2]{Sa2}
D. J. Saltman, \textit{Noether's problem over an algebraically
closed field}, Invent. Math. 77 (1984), 71--84.

\bibitem[Sa3]{Sa3}
D. J. Saltman, \textit{Retract rational fields and cyclic Galois
extensions}, Israel J. Math. 47 (1984), 165--215.

\bibitem[Sa4]{Sa4}
D. J. Saltman, \textit{Multiplicative field invariants}, J.
Algebra 106 (1987), 221--238.

\bibitem[Sa5]{Sa5}
D. J. Saltman, \textit{Multiplicative field invariants and the
Brauer group}, J. Algebra 133 (1990), 533-544.

\bibitem[Se]{Se}
J-P. Serre, \textit{Local fields}, GTM vol.67, Springer-Verlag,
New York, 1979.

\bibitem[Sh]{Sh}
I. R. Shafarevich,
\textit{On L\"uroth's problem},
Proc Steklov Inst. Math. 183 (1991), 241--246.

\bibitem[Sw1]{Sw3}
R. G. Swan, \textit{Induced representations and projective
modules}, Ann. Math. 71 (1960), 552--578.

\bibitem[Sw2]{Sw4}
R. G. Swan, \textit{Invariant rational functions and a problem of
Steenrod}, Invent. Math. 7 (1969), 148--158.


\bibitem[Sw3]{Sw1}
R. G. Swan, \textit{Noether's problem in Galois theory}, in ``Emmy
Noehter in Bryn Mawr", edited by B. Srinivasan and J. Sally,
Springer-Verlag, 1983, Berlin.

\bibitem[Sw4]{Sw2}
R. G. Swan, \textit{The flabby class group of a finite cyclic
group}, in ``Proceedings of The 4th International Congress of
Chinese Mathematicians, Hangzhou, 2007", edited by Lizhen Ji,
Kefeng Liu, Lo Yang and Shing-Tung Yau, Higher Education Press
(Beijing) and International Press (Somerville), 2008.




\bibitem[Ta]{Ta}
K. Tahara, \textit{On the second cohomology groups of semidirect products},
Math. Z. 129 (1972), 365--379.

\bibitem[Vo]{Vo}
V. E. Voskresenskii, \textit{Algebraic groups and their birational
invariants}, Transl. Math. Monographs vol. 179, Amer. Math. Soc.,
1998, Providence.


\end{thebibliography}
\end{document}